\documentclass[11pt,a4paper]{article}

\usepackage{geometry} 
\usepackage{amsmath,amsfonts,amsthm,amssymb, mathtools, mathrsfs} 
\usepackage[dvipsnames]{color, xcolor}  
\usepackage{enumitem} 
\usepackage{graphicx, float, subfig, overpic} 
\usepackage{titlesec} 
\usepackage[none]{hyphenat} 
\usepackage[font=small,labelfont=bf,tableposition=top]{caption}
\usepackage{todonotes, comment} 
\usepackage[normalem]{ulem}
\renewcommand{\emph}[1]{\textit{#1}}

\usepackage{hyperref} 
\hypersetup{
	colorlinks=true,
	linkcolor=blue,
	citecolor=blue,
}

\def\tp{{\tt p}}

\def\cK{\mathcal{K}}
\def\tC{{\Gamma}}
\def\ta{{\tt a}}
\def\tnu{{\tt v}}

\def\cU{{\mathcal U}}
\def\cUr{{\mathcal U}_{\sigma/2}^{\R}}
\def\cV{{\mathcal V}}

\def\cM{{\mathcal M}}
\def\cB{{\mathcal B}}
\def\cA{{\mathcal A}}
\def\cC{{\mathcal C}}

\def\R{\mathbb R}
\def\Z{\mathbb Z}
\def\sleq{\lesssim}
\def\sgeq{\gtrsim}
\def\betas{{\beta}}
\def\bC{\mathbb C}

\definecolor{greenmunsell}{rgb}{0.0, 0.66, 0.47}

\definecolor{beatrice}{rgb}{1.0, 0.0, 0.5}
\definecolor{santiago}{rgb}{0.0, 0.0, 1.0}
\definecolor{darkGreen}{RGB}{0, 200, 0}
\numberwithin{equation}{section}
\newcommand{\ombar}{\overline{\omega}}

\theoremstyle{plain}
\newtheorem{lemma}{Lemma}[section]
\newtheorem{corollary}[lemma]{Corollary}
\newtheorem{proposition}[lemma]{Proposition}
\newtheorem{remark}[lemma]{Remark}

\newtheorem{definition}[lemma]{Definition}

\newtheorem{theorem}[lemma]{Theorem}

\newcommand{\st}{\,:\,}
\newcommand{\norm}[1]{\left\lVert#1\right\rVert}
\newcommand{\vabs}[1]{\left| #1 \right|}

\newcommand{\claus}[1]{\left\{#1\right\}}

\newcommand{\dist}[2]{\mathrm{dist}\left(#1 \,, #2\right)}
\newcommand{\diam}[1]{\mathrm{diam}\left(#1\right)}

\newcommand{\wt}{\widetilde}

\newcommand{\reals}{\mathbb{R}}
\newcommand{\naturals}{\mathbb{N}}

\newcommand{\torus}{\mathbb{T}}

\newcommand{\integers}{\mathbb{Z}}

\newcommand{\e}{\varepsilon}

\newcommand{\MM}{\mathcal{M}}

\definecolor{myGreen}{RGB}{0, 200, 0}
\definecolor{myOrange}{RGB}{255, 100, 0}
\definecolor{myYellow}{RGB}{255, 200, 0}
\definecolor{myBlue}{RGB}{0, 200, 255}
\definecolor{myPurple}{RGB}{126, 47, 142}

\title{Nekhoroshev Theorem for time quasiperiodic perturbations of
	P-Steep systems} 
\author{Dario Bambusi\footnote{Dipartimento di Matematica, Universit\`a degli Studi di Milano, Via Saldini 50, I-20133
		Milano { \tt{email:dario.bambusi@unimi.it}}}, Santiago Barbieri\footnote{Departament d'Informàtica, Matemàtica Aplicada i Estadística, Universitat de Girona, Carrer de la Universitat de Girona 6, E-17003
		Girona { \tt{email:santiago.barbieri@udg.edu}}}, Mar Giralt\footnote{LTE. Observatoire de Paris - Université PSL, Sorbonne Université, 77 Avenue Denfert-Rochereau, F-75014 Paris { \tt{email:mar.giralt@obspm.fr}}}, Beatrice Langella\footnote{Dipartimento di Matematica, Universit\`a di Pisa, Largo Bruno Pontecorvo 5, I-56127 Pisa \tt{email: beatrice.langella@unipi.it}}}

\date{\today}

\begin{document}
	
	\maketitle
	\begin{abstract}
		We prove a Nekhoroshev type result for a time
                quasiperiodic perturbation of an integrable
                Hamiltonian system. More precisely, we assume that the
                integrable part is analytic and fulfills a generic
                nondegeneracy condition introduced by Nekhoroshev and
                called P-Steepness.  We add a small per\-tur\-bation
                which depends in a quasiperiodic way on time (with
                Diophantine frequency) and prove that -- for times
                exponentially long with the inverse of the size
                $\varepsilon$ of the perturbation -- the actions of
                the unperturbed system remain approximately
                constant. The proof is based on an extension to the
                time dependent case of the proof  {of classical Nekhoroshev's theorem} given by Guzzo,
                Chierchia and Benettin, which however requires new
                ideas in order to deal with the more complex geometry
                of resonances of the time dependent case.
	\end{abstract}
	
	
	\section{Introduction}\label{intro}
	In this paper we prove that the solutions of a generic integrable
	Hamiltonian system subject to a perturbation depending
	quasi-periodically on time are stable in the sense of Nekhoroshev: the
	actions associated to the integrable system stay in a small
	neighborhood of their initial value for times that are
	exponentially long with the inverse of the size of the perturbation.
	
	To state our result more precisely, for any pair of positive integers $n,m\ge {1}$, we consider a real-analytic Hamiltonian of the form 
	\begin{align}\label{hamiltoniana}
		H(p,q,t) = h_0(p) + \e V(p,q,\nu_1 t,...,\nu_m t)\,,
	\end{align}
	where $(p,q)\in \cU\times \mathbb T^n$ are action-angle coordinates, 
	$\cU$ is an open domain of $\mathbb R^n$, 	$\mathbb
        T:=\R/2\pi\Z$, $(\nu_1 t,...,\nu_m t)\in
	\mathbb T^m$ and {$0\leq \varepsilon\ll 1$ is a small parameter}. We  assume $\nu =(\nu_1, \dots, \nu_m)$ to be Diophantine.

	Systems of this kind are often used as a paradigm for the study of the
	stability properties of the dynamics under an external driving force
	(see e.g. \cite{Chi79}) and  appear in approximate models
	in plasma physics \cite{Kro80} and in celestial mechanics
	\cite{BenettinFassoGuzzo1998,Pinzari2013}. They also constitute an interesting model for purely
	theoretical investigations of the dynamics \cite{LMS03}. 
	
	In the time-independent case (namely $m=0$), Nekhoroshev's
        theory has been the object of a very large number of studies
        (see e.g.
        \cite{Nek1,Nek2,BGG,Galla,Loc92,LocNei,Gio,GCB,BounemouraNiederman2012}),
        and it is known that if $h_0$ satisfies a generic
        transversality condition on its gradient known as Steepness,
        then Nekhoroshev's stability estimates hold. Convex and
        quasi-convex functions are the easiest examples of steep
        functions; however convexity and quasiconvexity are only open properties, while
        Steepness is generic \cite{Nekhoroshev1973,
          Barbieri2025}. Moreover, steep non-convex Hamiltonians do
        appear in models, e.g. in celestial mechanics close to the
        Lagrangian points and in the $N$-body problem close to zero
        eccentricities and inclinations
        \cite{BenettinFassoGuzzo1998,Pinzari2013}.

	Coming to time dependent perturbations, Nekhoroshev himself
        already observed in \cite{Nek1} that his stability result
        applies also to the case when the Hamiltonian $h_0$ is subject
        to a \textit{time periodic} perturbation (namely
        \eqref{hamiltoniana} with $m=1$),
        provided it satisfies, instead of
          Steepness, a different generic non-degeneracy condition,
          that he called P-Steepness (see Definition
\ref{def:p.steep} below).

The case of truly time quasiperiodic perturbations, namely $m\ge 2$,
was first studied by Bounemoura in \cite{BouTime} in the setting where
the forcing frequency is Diophantine and the unperturbed Hamiltonian
$h_0$ is convex: he proved that Nekhoroshev's stability actually does
hold even in that framework. 
	
In the present paper, we tackle the case of a truly quasiperiodic
perturbation and we remove the restriction to convex Hamiltonians
$h_0$:  we prove that, under the same non-degeneracy property
considered in \cite{Nek1}, namely P-Steepness, a Nekhoroshev type theorem still holds, with
stability times that reduce to those of Bounemoura in the convex case
(see Theorem \ref{main} for a precise statement).

	We emphasize that our result is, to the best of our knowledge, the first one providing exponential stability estimates that are robust under time quasiperiodic perturbations and hold under generic assumptions on the unperturbed system.
	
	\vspace{12pt}
	
The proof we give is different from Bounemoura's one, even if we
borrow some ideas from his construction. Indeed we use a variant of
Nekhoroshev's construction as improved by Guzzo, Chierchia and
Benettin in \cite{GCB}, which is known to provide
  the sharpest exponents available in the literature for the steep
  case. However, dealing with the time dependent case requires new
  ideas, which we are now briefly going to sketch.
  To explain them, we
first remark that, as it is often the case in the study of
nearly-integrable Hamiltonian systems, it is necessary to have a
quantitative control on ``small denominators", i.e. the
commensurability conditions that the frequencies of the unperturbed
system may satisfy in correspondence of the so-called
``resonances". Nekhoroshev's proof \cite{Nek1, GCB} is based on a
detailed analysis of the geometric structure of resonant zones, namely
the regions in the action space where the small denominators are
actually small.  In the time dependent case the small
denominators which appear are of the form
	\begin{equation}
		\label{risonn}
		\omega(p)\cdot k+\nu\cdot l\ ,\quad (k,l)\in\Z^{n+m}\ ,
	\end{equation}
where $\omega(p):=\nabla h_0(p)$ and the resonances are
classified by the maximal  {sublattices} $\Lambda$ of $\Z^{n+m}$ to which
the vectors $(k, l)$ belong. {A natural attempt could be to extend the original phase space by
adding to the Hamiltonian $H$ in \eqref{hamiltoniana} a term $\nu
\cdot J$, with $J \in \R^m$ a vector of variables conjugated to the
angles $(\phi_1, \dots, \phi_m) \equiv (\nu_1 t, \dots, \nu_m t)$, and
then to perform the geometric analysis of \cite{Nek1,GCB} in such extended
action space $(p, J)$. We immediately notice that this is not a
successful strategy, essentially due to the fact that non-degeneracy
conditions such as Steepness and its variants are based on the
property that, moving the actions $(p,J)$, the small denominators
change, whereas the
small divisors appearing in \eqref{risonn} are independent of the
extended action variables $J$,  {therefore they} do not move as $J$ vary. Thus
we work directly in the action space of the original system {, we analyze}
the structure of the resonant domains and study how to modify
Nekhoroshev's construction in order to deal with the quasiperiodic
forcing.}

\vspace{10pt}
	
 {We now present the main ideas of our approach and the difficulties involved.} Consider first the case where $p$ is such that the quantity
\eqref{risonn} is small only for $(k,l)$ varying in a one dimensional
 { sublattice} $\Lambda  {\subset \Z^{m+n}}$. The way we proceed consists in constructing a
vector $\ombar$ such that
\begin{equation}\label{frequency.space}
	\omega\cdot k+\nu\cdot l= (\omega-\ombar)\cdot k\ ,\quad
	\forall (k,l)\in\Lambda\ \,,
\end{equation} 
 {with $\omega \equiv \omega(p)$.} {We remark that in the time independent case $\ombar=0$.}
 {One thus sees that in frequency space, unlike in the time independent case, the resonant plane does not pass through the origin, but it is instead located at a finite distance from it.} We call the vector $\ombar$ center of the resonant plane
corresponding to $\Lambda$. Of course two resonant planes
corresponding {to two different} resonant moduli $\Lambda_1 \neq
\Lambda_2$ can intersect, 
giving rise to resonances of order two.  Now, in the time independent
case resonant planes can only intersect at subspaces passing through
the origin of the frequency space,
while in the time dependent case
there is a dense countable set of points $\bar \omega$ such that for
any $\ombar$ in such set, the resonant planes
intersect at affine subspaces {passing through $\ombar$}.
This leads to a much more intricate geometry of resonances, which in turn makes more delicate to rule out overlapping of resonances phenomena.


We resolve this issue by first taking advantage of the fact that,
  since $\nu$ is Diophantine, then at any finite order  {we can quantitatively bound from below the distance among different points $\ombar$.}
  The reason being
  that if two vectors $l,l'$ are s.t. $\nu\cdot l\simeq\nu\cdot l'$
  then by the Diophantine property  {the quantity} $\|l-l'\|$ must be large.
  
  Then the geometric
  construction is based on the remark that close to every center
  $\ombar$, the situation is similar to the one occurring at 
  the origin in the time independent case. We {observe}
  that in the time periodic case the set of the possible centers
  $\ombar$, instead of being asymptotically dense, forms a lattice:
  that is why the periodic case is much simpler.
	
	\vspace{12pt}

We conclude this introduction by remarking that the case of a
quasiperiodic forcing appears also when studying instability
phenomena, in particular when one considers the splitting of the
invariant manifolds of unstable {tori} in models inspired by
Arnol'd's classical construction \cite{LMS03}. Indeed, it is well
known that, in general, if the invariant manifolds intersect
transversally, then the sizes of their splitting angles are related to
the time of instability of a specific orbit. Therefore it is
particularly relevant to understand if the system \eqref{hamiltoniana}
is stable and to determine its time of stability.

\vspace{12pt}

	The rest of the paper is organized as follows: in
	Sect. \ref{set} we give a precise statement of our stability
	result, and we discuss it. In particular we examine the value of the
	stability exponents that we get and we compare them to those previously  obtained
	in different contexts. In Sect. \ref{geometric} we study the
	geometry of the resonances in phase space. This is the heart of the
	paper. In Sect. \ref{analytic} we put the system in resonant normal
	form in each one of the resonant domains constructed in
	Sect. \ref{geometric}. Actually we just use the main analytic lemma of
	\cite{poe}, which has a form suitable for our purpose. In
	Sect. \ref{dyn.arg} we use Nekhoroshev's idea that a solution can only
	exit from a resonant zone by entering in a zone fulfilling lower order
	resonances, thus allowing to conclude that after a finite number of
	resonance losses, the solution has to stop.  
	
	\vspace{10pt}
	
	\noindent{\it Acknowledgments.} {D. Bambusi, M. Giralt, and B. Langella have been supported by the research projects PRIN 2020XBFL 
	``Hamiltonian and dispersive PDEs'' of the Italian Ministry of
        Education and Research (MIUR). D. Bambusi has been supported by GNFM. B. Langella has been supported by GNAMPA.
      M. Giralt has been  supported by the Seal of Excellence program 2025 of Sorbonne Université. {S. Barbieri has been partially supported by grant RL001607 of Professor M. Guàrdia,
      funded by the Catalan Institution for Research and Advanced Studies (ICREA),
      and by the Juan de la Cierva fellowship JDC2023-052632-I. }} 
	
	\section{Setting and main result}\label{set}
	
	We give here the precise assumptions on the Hamiltonian \eqref{hamiltoniana}, the first one being that $\nu =(\nu_1,\dots,\nu_m)$ is a Diophantine vector,
	that is,
	
	\medskip 
	
	\noindent{\bf Hypothesis 1:} there exist $\tC>0$ {and
          $\tau\geq m-1$} such that
	\begin{equation}\label{eq:nuDiophantine}
		\vabs{\nu \cdot l} \geq \frac{\tC}{\|l\|^{\tau}}, 
		\qquad 
		\forall l \in \integers^m \setminus\claus{0},
	\end{equation}
	{where $ \left\|x\right\| $ indicates the standard
	Euclidean norm}.
	
	{{Secondly}, the Hamiltonian} $H$  {in \eqref{hamiltoniana}} is assumed to be analytic. {Namely, for any $ {d}\in \mathbb N$, and for
	$x\in\bC^d$ we denote by $\mathcal{B}_r(x)\subset \bC^d$ the
 complex open ball of radius $r$ and center $x$ associated to the norm $\left\|x\right\|$.}  Given a set $\cV\subset \bC^d$ and a parameter $r>0$, we indicate by
	\begin{equation}
		\label{extendion}
		\cV_r:=\bigcup_{x \in\cV}\mathcal{B}_r(x)\ ,\quad
		\cV_r^{\R}:=\cV_r\cap \R^d
	\end{equation}
	its complex extension of width $r$ and its real projection, respectively. \\
	Then, we assume that
	
	\medskip 
	
	\noindent{\bf Hypothesis 2:} there exists an open domain $\cU\subset\R^n$ and a
	positive parameter $\sigma$ s.t. $h_0$
	extends to a complex analytic function on $\cU_\sigma$ and $V$ extends
	to a complex analytic function on
	$\cU_\sigma\times\torus^n_\sigma\times \torus^m_\sigma$.

	\medskip

	We now recall the definition of P-Steepness {given in} \cite{Nek1}:
	\begin{definition}\label{def:p.steep}
		Consider an open domain $\mathcal{V} \subseteq \R^n$, a function  $h  \in
		C^2(\cV^{\R}_\sigma  , \reals)$, and let $\omega:= \nabla h$ be its gradient.
		
		\begin{itemize}
			\item[(i)] $h$ is said to be \emph{P-steep} at
			$p_0\in\cV$ if there exist a radius $r>0$, coefficients
			$D_1, \dots, D_n>0$ and exponents $\alpha_1, \dots, \alpha_n
			\geq 1$ such that for any {linear subspace} $M${$\subseteq \R^n$} of dimension
			$s=1,\dots, n$, one has
			\begin{equation}\label{p.steep.p0}
				\max_{\eta \in (0, \xi]}\, \min_{\substack{\|u\| = 1\\u \in M}} \left\|\Pi_M \left(\omega (p_0 + u \eta) - \omega(p_0)\right) \right\| > D_s \xi^{\alpha_s} \quad \forall \xi \in (0, r)\,.
			\end{equation}
			\item[(ii)] The function $h$ is said to be \emph{P-steep} on
			$\mathcal{V}$ if $h$ is P-steep at $p_0$
			for any $p_0 \in \mathcal{V}$, with uniform
			radius, coefficients and exponents, namely if there exist
			$r>0$, $D_1, \dots, D_n >0$, and $\alpha_1,\dots, \alpha_n
			\geq 1$ such that $h$ satisfies \eqref{p.steep.p0} for any
			$p_0 \in \mathcal{V}$.
		\end{itemize}
	\end{definition}

	We observe that convex and quasiconvex functions are P-steep 
	with
	\begin{equation}
		\label{alpha1}
		\alpha_1=...=\alpha_n=1\ .
	\end{equation}
	 
We also remark that convexity {and quasiconvexity} are open
properties, which however are not dense. {Concerning P-Steepness
  (and also Steepness)
  the following ``genericity type'' property was proved in
\cite{Nek1}.}

{
\begin{theorem}
\label{P.generic} [Theorem 1.13.B of \cite{Nek1}]
If $h$ is not $P$-steep in any neighborhood of a
point $p_0\in\cV$ then it is infinitely degenerate: the coefficients in
the Taylor series of $h$ about this point satisfy infinitely many
independent algebraic equations.
\end{theorem}}

For standard Steepness one can find a modern demonstration in
\cite{Barbieri2025}. The proof and the results of \cite{Barbieri2025}
can be easily adapted to the P-steep case. The arguments that are
used there essentially rely on real-algebraic geometry and complex
analysis. In particular, the core element is the analytic version of a
theorem due to Yomdin and Gromov \cite{Gro_1987,Yom_2008, Bur_2008}
which ensures that, for any semi-algebraic set{\footnote{Broadly
  speaking, a semi-algebraic set of $\R^n$ is determined by a finite
  number of polynomial equalities and inequalities}}, there exists a
collection of $C^k$-mappings which parametrize the considered set.

We also remark that the original definition of Steepness is quite involved,
but for real-analytic functions the following equivalent
characterization based on arguments of semi-analytic geometry holds:
$h$ is steep iff it has no critical points and if its restriction to any
proper affine subspace admits only isolated critical points
\cite{Niederman}. We think that an analogous characterization should
hold also for P-steep functions. {Ultimately, therefore, we make the last following}

	\medskip 
	
	\noindent{\bf Hypothesis 3:} $h_0$ is P-steep on $\cU_{\sigma/2}^\R$. 
	
	\medskip

	After these premises, we are ready to state our main result
	
	\begin{theorem}
		\label{main}
		Under the Hypotheses 1, 2 and 3 above, there exist strictly positive finite
		constants $\varepsilon_*$, $a$, $b$, $C$, $C_*$ such that if
		$0\leq\varepsilon<\varepsilon_*$, then for any initial datum
		$(p_0,q_0)\in\cU\times\torus^n$, one has 
		\begin{equation}
			\label{main.esti}
			\left\| p(t)-p_0\right\|\leq C
			\varepsilon^{b} \quad \text{for any }\quad 
			|t|\leq \frac{1}{C {\varepsilon^{\frac 1 2}}}
			\exp\left(\frac{C_*}{\varepsilon^{a}}\right)\ .
		\end{equation}
		Furthermore, defining
		\begin{equation}
			\label{tp}
			\tp:=\prod_{i=1}^{n}\alpha_i\ ,\quad \tp_1:=\prod_{i=1}^{n-1}\alpha_i
		\end{equation}
		one can take
		\begin{equation}
			\label{le.costanti}
			a=\frac{1}{2\tp_1  (n+1) (\tau+ 1) }\ ,\quad
			b{=\frac{1}{2\tp}\left(
				1-\frac{n}{   (n+1) (\tau+ 1)}\right)}\ \ .  
		\end{equation}
	\end{theorem}
	
	We add a few comments.
	
As we already emphasized, while exponential stability is well known in
the case of time independent perturbations and generic unperturbed
system, for the case of time quasiperiodic perturbations it was only
known when the unperturbed Hamiltonian is convex, which is a non
generic property. Here we extend the result to the case of unperturbed
generic system.
	
{We remark that, exactly as the result by Nekhoroshev for the time
  periodic case and the result by Bounemoura for the quasiperiodic
  convex case, our result can be extended to the case of Hamiltonians
  of the form
		\begin{equation}
			\label{molto.estese}
			h_0(p)+\nu\cdot J+ \varepsilon V(p,J,\phi,\psi)\ ,
		\end{equation}
in which one extends the phase space and the perturbation also depends
on the auxiliary actions. More precisely, our proof can be adapted
with minor changes to this case. The main point is that the geometric
construction has to be done only in the space of the actions $p$.}
{For the interest of such an extension see \cite{BouTime}.}

We discuss now the value of the {\it stability exponents} $a$ and
$b$. We recall that the exponent $a$ controls the stability times of
the system and for this reason its value has been the object of a
quite intense investigation at least for the time independent
case. The original proof by Nekhoroshev gave a value which in the
convex case reduces to $a\simeq n^{-2}$. Lochak's proof \cite{Loc92}
allowed to improve considerably the exponent in the quasiconvex case
leading to a value of $a=(2n)^{-1}$ (see \cite{LocNei}). Such an
exponent was obtained also through arguments close to Nekhoroshev's
ones, but only in the quasiconvex case by P\"oschel \cite{poe}. A
similar value has been obtained also in some infinite dimensional
models like an infinite chain of coupled rotators or in some nonlinear
PDEs in \cite{BG93,Bam99,BG24}, but in this case $n$ is
the number of degrees of freedom actually excited in the initial
datum.
	
	For the non convex (and also non quasiconvex) case, in order to get a
	better exponent one has to wait until the paper \cite{GCB} in which
	Guzzo Chierchia and Benettin modified the proof by Nekhoroshev
	obtaining for the steep case an exponent $a=1/(2n\alpha_1\dots\alpha_{n-2})$.
	
	We remark that in the convex case the best known exponent, which is
	considered essentially optimal, has been obtained by Zhang and Zhang
	\cite{Zhang2011,ZZ17} and it is given by
	\begin{equation}
		\label{zhang}
		a=\frac{1}{2(n-2)}-\mu\ ,
	\end{equation}
	with $\mu$ an arbitrary small number.
	
	We also recall that Nekhoroshev type results hold also in the case of
	less regular systems \cite{MS03, BouCk, Bou_preva, nek_noi} and even in the case of H\"older
	perturbations \cite{BarMarMas}, but the stability times that one gets are
	only polynomial in $\varepsilon^{-1}$. 
	
Returning to the time-dependent case, the story begins with a
conjecture originally formulated by Chirikov \cite{Chi79} and later made precise in
\cite{LMS03} according to which, in the convex case, provided $\nu$ is
Diophantine of type $(\Gamma,\tau)$ (cfr. \eqref{eq:nuDiophantine}),
one should have
	\begin{equation}
		\label{chiloc}
		a=\frac{1}{2(n+1+\tau)}\ .
	\end{equation}
For the quasiperiodic convex case Bounemoura \cite{BouTime} got  {exponents}
	\begin{equation}
		\label{boune}
		a=\frac{1}{2(n+1)(\tau+1)},\quad  b=\frac{1}{2}
		\left(1-\frac{n}{(n+1)(\tau+1)}\right)\ .
	\end{equation}
As Bounemoura observed, while $a$ is quite far from the value
conjectured by \cite{LMS03}, the exponent $b$ is much better
than what one cloud expect. Heuristically speaking, this is due to the
fact that in the quasiperiodic case one adds ``more frequencies" and
therefore ``more resonances" to the system. Indeed, as Nekhoroshev
stability results essentially works by preventing resonances of the
same order to overlap (see e.g. the discussions in \cite{GCB}), this
means that the resonant zones must be ``smaller" than in the
time-independent case, and this reflects in having a confinement in a
smaller neighborhood. This, in turn, has a consequence on the time of
stability, that must be shorter.

	The exponents \eqref{le.costanti} that we get reduce
	exactly to those of Bounemoura in the convex case.  Actually we do not
	know if such exponents are optimal or not, but we are convinced that a
	technique based on the analysis of the geometry of resonances in the
	action space would not lead to the exponents \eqref{chiloc}, since the
	structure in which the product of $n$ and $\tau$ appears is essentially
	due to the fact that we have to estimate from below the distance among
	different centers  {$\ombar$ (cfr. Eq. \eqref{frequency.space})}, which is unavoidable in our construction. A possible idea in order to reach the exponents predicted by Chirikov is explained in the paper \cite{BouTime} by Bounemoura and amounts to giving a proof of Nekhoroshev's theorem based on badly approximable vectors. However, no result in this direction currently exists in the literature.

	Finally we add a comment on 
	P-Steepness: the fact that the resonant planes intersect also at
	planes not passing through the origin is the reason why we need
	P-Steepness instead of Steepness. Indeed, Steepness amounts
	essentially to Definition \ref{def:p.steep} in which
	\eqref{p.steep.p0} is substituted by
	\begin{equation*}
		\max_{\eta \in (0, \xi]}\, \min_{\|u\| = 1\,, u \in M}
		\left\|\Pi_M \omega (p_0 + u \eta) \right\|
		\geq D_s \xi^{\alpha_s}\ , \quad \forall \xi \in (0, r)\,
	\end{equation*}
	and $h_0$ is assumed not to have any critical point.

	\section{Geometric construction}\label{geometric}

	In this section we define a partition of the action space $\cUr$ into
	blocks which are essentially left invariant by the flow of a
	Hamiltonian in normal form.
	
	\subsection{Preliminaries}
	
	Following Bounemoura we give the following definition. Let $K$ be a
	large parameter whose value will be fixed in Sect.\ref{analytic}.
	\begin{definition}\label{def:resonantmodule}
		A module $\Lambda\subset\Z^{n+m}$ is said to be a $K$-submodule if it
		is generated  {over the integers} by elements $(k,l)\in\Z^n\times\Z^m$ such that  {$\|k\| + \|\ell\| \leq K$,}
		 and it is said to be a  {$K$-maximal submodule if it is
		not properly contained in any other $K$-submodule of the same
		dimension. A $K-$maximal submodule} $\Lambda$ is said to be admissible
		if its intersection with $\left\{0\right\}\times\Z^m$
                is {the trivial set $\{(0,0)\}$.}
		
		The set of admissible  {$K$-maximal} submodules with rank
		 $s$ will be denoted by $\cM^{(s)}_K$.
	\end{definition}
	
	 {Let $\Lambda\in\cM^{(s)}_K$, then it admits a basis (as a modulus) $\{v_j\}_{j=1}^s$ {that decomposes in $v_j = (k_j,  l_j) \in \Z^n \times \Z^m$ and satisfies $\|k_j\|+\|l_j\|\leq K$.}  Then, by admissibility, the vectors
	$\left\{k_j\right\}_{j=1}^s$ are {linearly} independent and generate a submodule
	of $\Z^n$ of rank $s$, that will be denoted by $\tilde \Lambda$.}  {Note that this also immediately implies that $0 \leq s \leq n$.}
	
	 {Let furthermore $\cK = [k_1 \dots k_s]$ be the matrix with
           the components of the vectors $k_j$ as columns, then $\cK^t
           \cK$ is the Gram Matrix of the lattice generated by the
           vectors $k_j$ and we denote
	\begin{equation}
		\label{volume}
		\left|\tilde\Lambda\right|:=\sqrt{\det(\cK^t \cK)}\,, 
	\end{equation}
which is the volume of the fundamental parallelepiped of the lattice
generated by $k_1, \dots, k_s$.

	{From now on, we will denote by
	$\tilde\Lambda_\R:=span_\R\left\{k_1,...,k_s\right\}$ the plane
	generated by the real linear combinations of the generators $k_1,\dots,k_s$, and we will indicate by
	$\Pi_{\tilde\Lambda}$ the orthogonal projector onto $\tilde\Lambda_{{\R}}$.}

	\subsection{Resonant planes and their centers}\label{centers}

	\begin{definition}\label{def:risonanza_esatta}
		Let $s\in\{1,..,n\}$ and $\Lambda\in\cM^{(s)}_K$. {The 
		\textbf{resonant plane} associated to $\Lambda$ is the hyperplane}
		\[
		P_\Lambda = \claus{\omega\in \reals^n \st
			\omega\cdot k + \nu \cdot l = 0, \, \forall \ (k,l)\in\Lambda}.
		\]
	\end{definition}
	
	{Now, consider} $\Lambda\in\cM^{(s)}_K$, and let $v_j\equiv(k_j,l_j)$, $j=1,...,s$ be a  {set of linearly independent vectors in} $\Lambda$.
	Then we consider the basis $u_j$ of $\tilde\Lambda_{\R}$ dual to
	$(k_1,...,k_s)$, which is
	defined by
	\begin{equation}
		\label{def.dual}
		k_i\cdot u_j=\delta_{ij}\ ,\quad  \forall i,j=1,...,s\ ,\quad u_j\in
		\tilde\Lambda_{\R}\ .
	\end{equation}
	We use such a basis to define
	\begin{equation}
		\label{defom.bar}
		\ombar\equiv \ombar_{\Lambda} := -\sum_{j=1}^s (\nu
		\cdot l_j)u_j .
	\end{equation}
        {
          \begin{lemma}
          \label{basis}
The vector $\ombar_{\Lambda}$ is independent of the basis of $\Lambda$.
        \end{lemma}
        }
{\proof Let $v'_1 = (k'_1, l'_1), \dots, v'_s = (k'_s, l'_s)$ be a
basis of $\Lambda$ different from the one used to construct
$\ombar_\Lambda$, then there are  coefficients $\{\alpha_{i,j}\}_{i, j = 1}^s$  such that $v'_j = \sum_{i} \alpha_{i,j} v_i$, then a direct computation {yields}
		$$
		\ombar \cdot k'_j = \sum_{i=1}^s \alpha_{i,j} \ombar \cdot k_i = -\nu \cdot \sum_{i=1}^s \alpha_{i,j} l_i = -\nu \cdot l'_j\,, \quad \forall j = 1 \dots s\,.
		$$}
\qed

	One has the
	fundamental property 
	\begin{equation}
		\label{def.om.1}
		\omega \in P_\Lambda
		\quad\iff\quad  
		\Pi_{\tilde\Lambda}(\omega - \ombar)= 0\ .
	\end{equation}
	
	\begin{definition}
		\label{def.centre}
		The vector $\ombar_\Lambda$ will be called the {\bf
			center of the resonant plane $P_\Lambda$.}
	\end{definition}

	\subsection{Resonant zones}\label{rzone}
	
	In this section we construct a covering of $\cUr$ labeled by the
	resonance moduli $\Lambda$. {The covering will depend on}
{suitable}  sets of real positive parameters
	$\left\{A_s\right\}_{s=0}^n$, $\left\{C_s\right\}_{s=0}^n$,
	$\left\{d_s\right\}_{s=0}^n$ on which we will impose several
	conditions along the section. {The relationship among their elements will be fixed
	in Subsection \ref{le.costanti.s}, while their values will be
	set in Section \ref{analytic} so to obtain the best exponents
	in the stability estimates.}

	From now on, we assume that $h_0$ in \eqref{hamiltoniana} is
        P-steep on $\cUr$
{and that, without any loss of generality, its coefficients satisfy $D_1, \dots, D_n \in (0, 1)$.}
    {We also set}
	\begin{equation}\label{sup.omega}
	\cC_\omega := \sup_{p \in \cUr} \|\partial^2_p h_0 (p)\|\,.
	\end{equation}

	The next definition we give is meant to identify the points $p$ which
	are in resonance with  {vectors in} a module in $\cM^{(s)}_K$.

	\begin{definition}
		Let $K\geq 1$, $s\in\{1,..,n\}$ and $\Lambda \in \MM_K^{(s)}$.
		Then we define the \textbf{resonant zone} of module $\Lambda$ as
		\begin{equation}\label{beta.zeta}
			Z_{\Lambda}^{(s)} :=
			\left\lbrace p\in\cUr \ |\ \norm{\Pi_{\tilde \Lambda} (\omega(p) - \ombar_{\Lambda})}< \beta^{(s)}_{\Lambda}\right \rbrace\,, \quad \beta^{(s)}_\Lambda := \frac{C_s}{K^{A_s}|\tilde\Lambda|}\,.
		\end{equation}
		Moreover, we define
		\begin{align}
			Z^{(s)}:=\bigcup_{\Lambda\in\cM^{(s)}_K}Z^{(s)}_\Lambda\ .
		\end{align}
	\end{definition}

	Lastly, to obtain a covering of $\cUr$, we define the non-resonant regions:
	\begin{align*}
		Z^{(0)}:= 
		\cUr \setminus Z^{(1)}
		\,.
	\end{align*}


{We now define the resonant blocks which contain the points which
  are resonant with only one module.}
	
	\begin{definition}
		Let $K\geq 1$, $s\in\{1,..,n\}$ and $\Lambda \in \MM_K^{(s)}$.
		We define the \textbf{resonant block} of modulus $\Lambda$ as
		\begin{align*}
			& B_{\Lambda }^{(s)} := Z_\Lambda ^{(s)} \setminus
			Z^{(s+1)}\ ,\quad s=1,...,n-1
			\\
			& B^{(n)}_\Lambda :=Z^{(n)}_\Lambda \ .
		\end{align*}
		If $s=0$, we also define $B^{(0)} := Z^{(0)}$.
	\end{definition}
	
{Finally, we introduce another kind of sets that will turn out to be relevant when studying the dynamics.}
	
	\begin{definition}
		\label{def:fastDriftBlock}
		Let $\Lambda\in\cM^{(s)}_K$, with $s\in\left\{1,...,n\right\}$
		then, for $p \in B_{\Lambda}^{(s)}$, we define its \textbf{fast drift block} as
		\begin{equation}\label{def:salsicciotti}
			F_{\Lambda}^{(s)}(p) = \left[\claus{p + \tilde\Lambda_{\reals}}_{d_s} \cap Z_{\Lambda}^{(s)}\cap\R^n \right]^{p}\,,
		\end{equation}
		where $\cA+\cB$ is the Minkowski sum between sets, namely $\cA+\cB =
		\claus{a+b \st a\in \cA, \, b\in \cB}$, $[\cA]^p$ is the connected
		component of the set $\cA$ containing $p$, while as in
		Sect. \ref{set},  {$\cA_r$ is the union {of the complex}
		balls of radius $r$ centered at the points $x \in \cA$. }
		\\
		If $p \in B^{(0)}=Z^{(0)}$, we define
                \begin{equation}
                  \label{F.0}
		F^{(0)}(p) = \mathcal{B}_{d_{{0}}}(p)\cap\cUr\,.
                  \end{equation}
	\end{definition}

	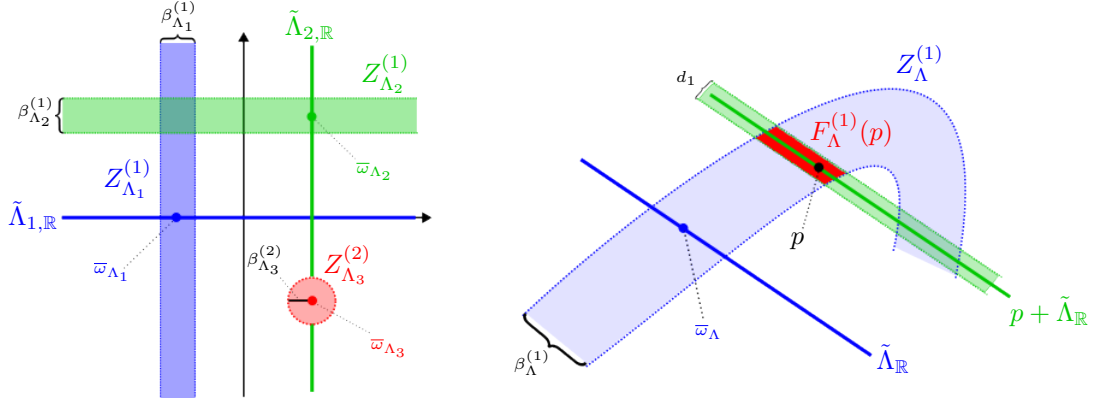
\begin{figure}
		\centering
		\begin{overpic}[scale=0.45]{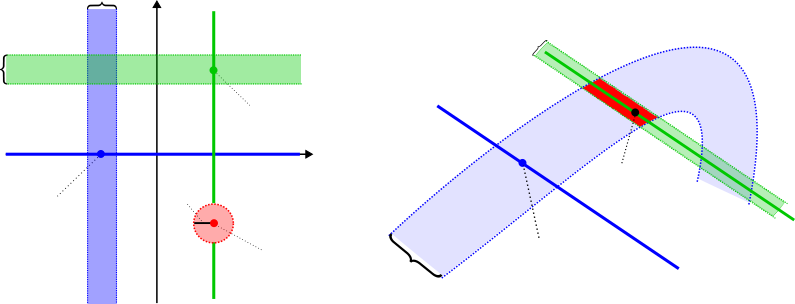}
			\put(4,13){\tiny \color{blue} $\overline{\omega}_{\Lambda_1}$}
			\put(11,39.5){\tiny $\beta^{(1)}_{\Lambda_1}$}
			\put(-5,18){\small \color{blue}$\tilde\Lambda_{1,\reals}$}
			\put(5,22){\small \color{blue} $Z^{(1)}_{\Lambda_1}$}
			\put(31.5,23.5){\tiny \color{darkGreen} $\overline{\omega}_{\Lambda_2}$}
			\put(-4,29.5){\tiny $\beta^{(1)}_{\Lambda_2}$}
			\put(24,38){\small \color{darkGreen}$\tilde\Lambda_{2,\reals}$}
			\put(32,33){\small \color{darkGreen} $Z^{(1)}_{\Lambda_2}$}
			\put(33,5){\tiny \color{red} $\overline{\omega}_{\Lambda_3}$}
			\put(20,14){\tiny $\beta^{(2)}_{\Lambda_3}$}
			\put(28,13){\small \color{red} $Z^{(2)}_{\Lambda_3}$}
			\put(48,3){\tiny $\beta^{(1)}_{\Lambda}$}
			\put(86,3){\small \color{blue} $\tilde{\Lambda}_{\reals}$}
			\put(67,6.5){\tiny \color{blue} $\overline{\omega}_{\Lambda}$}
			\put(65,33){\tiny  $d_1$}
			\put(77,16){\small $p$}
			\put(88,34){\small \color{blue} $Z^{(1)}_{\Lambda}$}
			\put(100,8){\small \color{darkGreen} $p+\tilde{\Lambda}_{\reals}$}
			\put(79,27){\small \color{red} $F^{(1)}_{\Lambda}(p)$}
		\end{overpic}
		\caption{Left: Representation {in the case $n=2$,
                    $m=1$} on the $\omega(p)$-plane of the resonant zones of module $\Lambda_1 = span_{\mathbb{Z}}\{(1,0;1)\}$, $\Lambda_2 = span_{\mathbb{Z}}\{(0,1;-1)\}$, $\Lambda_3 = span_{\mathbb{Z}}\{(1,0;-1),(0,1;1)\}$. Right: Representation on the $p$-plane of the fast drift plane $F_{\Lambda}^{(1)}(p)$, in red.}
	\end{figure}
	
	\begin{remark}
		It is immediate to see that resonant blocks and fast drift
		blocks are a covering of $\cUr$. Indeed, resonant zones are a
		covering, and by construction 
		$$
		B^{(0)}\bigcup\left(	\bigcup_{s = 1}^{n} \bigcup_{\Lambda\in\cM^{(s)}_K}   B^{(s)}_{\Lambda}\right) = B^{(0)}\bigcup\left(\bigcup_{s = 1}^{n} \bigcup_{\Lambda\in\cM^{(s)}_K}  Z^{(s)}_{\Lambda}\right)\,, 
		$$
		{so that} resonant blocks are still a covering. Since for any $s$,
		for any $\Lambda\in\cM^{(s)}_K$
		$$
		B^{(s)}_{\Lambda} \subset \bigcup_{p \in B^{(s)}_{\Lambda}} F^{(s)}_{\Lambda}(p)\,,
		$$
		one also has that {also} the fast drift blocks are a covering. 
	\end{remark}
	
	\begin{lemma}
		\label{lemma:propertiesResonantZones}
		For any $s\in\{0,..,n-1\}$, the following inclusion holds:
		\begin{equation}
			\label{inclusive}
			\cUr \setminus Z^{(s+1)} \subseteq \bigcup_{\substack{\Lambda' \in \MM_K^{(s')} \\ 0\leq  s'\leq s}} B_{\Lambda'}^{(s')}\,.
		\end{equation}
	\end{lemma}
	
	\begin{proof}
		The proof goes by induction on the dimension $s$. If $s=0$,
		$$
		\cUr \setminus Z^{(1)} = Z^{(0)} = B^{(0)}\,,
		$$
		therefore \eqref{inclusive} holds.
		Suppose \eqref{inclusive} is proven for all $s' \leq s$. 
		We recall that, by the very definitions of
		$Z^{(s+1)}_\Lambda$ {and} $ B^{(s)}_\Lambda$, one has
		$$
		Z^{(s)}\setminus Z^{(s+1)}=	\Bigg(\bigcup_{\Lambda \in \MM_K^{(s)}} Z_\Lambda^{(s)}\Bigg) \setminus Z^{(s+1)} = \bigcup_{\Lambda \in \MM_K^{(s)}} B^{(s)}_\Lambda\,.
		$$
		Then we use the trivial decomposition
		$\cUr = (\cUr \setminus Z^{(s)}) \cup Z^{(s)}$
		to get { formula \eqref{inclusive}, namely}
		$$
		\begin{aligned}
			\cUr \setminus Z^{(s+1)}& =  \left(\left(\cUr \setminus
			Z^{(s)} \right) \setminus Z^{(s+1)} \right) \cup \left(Z^{(s)}
			\setminus Z^{(s+1)} \right)
			\\
			&\subseteq \Bigg( \Big( \bigcup_{\substack{\Lambda \in \MM_K^{(s')} \\ 0\leq s' \leq s-1}} B^{(s')}_\Lambda \Big) \setminus Z^{(s+1)} \Bigg) \cup \Big( \bigcup_{\Lambda \in \MM_K^{(s)}} B^{(s)}_\Lambda \Big) \subseteq \Big(\bigcup_{\substack{\Lambda \in \MM_K^{(s')} \\ 0\leq s' \leq s}} B^{(s')}_\Lambda\Big)\,.
		\end{aligned}
		$$
	
	\end{proof}
	
	\begin{lemma}[Diameters]\label{lemma:diameters}
		Let $K \geq 1$, $s \in \{1, \dots n\}$, $\Lambda \in \MM^{(s)}_K$  and suppose {that} the parameter $d_s$ appearing in \eqref{def:salsicciotti} satisfies
		\begin{equation}\label{small.diam}
			d_s \leq \frac{\beta^{(s)}_\Lambda}{2
			\cC_\omega}\ ,\quad \text{if}\quad s=1,...,n \,,
		\end{equation}
		and also
		\begin{equation}
			\label{small.diam.1}
			d_s\leq \frac{1}{2} \left(\frac{4\beta^{(s)}_\Lambda}{D_s}\right)^{\frac{1}{\alpha_s}}\ ,\quad \text{if}\quad s=1,...,n-1\,.
		\end{equation}
		Then,	for any $p \in B^{(s)}_{\Lambda}$, one has
		\begin{equation}\label{diam.cost}
			\diam{F_{\Lambda}^{(s)}(p)} < 2 \left(\frac{4
				\beta^{(s)}_\Lambda}{D_s} \right)^{\frac{1}{\alpha_s}} \, ,
		\end{equation}
		{where} $D_s$ and $\alpha_s$ {are} the constants in the definition \ref{def:p.steep} of
		P-Steepness. 
	\end{lemma}
	\begin{proof}
		{Suppose, for contradiction,} that there are two points $p_1, p_2 \in F_{\Lambda}^{(s)}(p)$ such that $\|p_1 - p_2\| \geq  2 \left(\frac{4 \beta^{(s)}_\Lambda }{D_s} \right)^{\frac{1}{\alpha_s}}$. Now, { as $ F_{\Lambda}^{(s)}(p)$ is open and connected, it is connected by arcs.} This implies that there exists a curve $\gamma:[0,1] \rightarrow F_{\Lambda}^{(s)}(p)$ such that
		$$
		\gamma(0) = p_1\,, \quad \gamma(1) = p_2\,, \quad \gamma(t) \subseteq F_{\Lambda}^{(s)}(p) \quad \forall t\,.
		$$
		Since $\|\gamma(0) - p_1\| = 0$ and $\|\gamma(0) - p_2 \| \geq
		2 \left(\frac{4 \beta^{(s)}_\Lambda }{D_s}
		\right)^{\frac{1}{\alpha_s}},$ there exists a value $t_\star
		\in [0, 1]$ such that
		\begin{equation}\label{t.star}
			\| \gamma(t_\star) - p_1\| =  2 \left(\frac{4 \beta^{(s)}_\Lambda }{D_s} \right)^{\frac{1}{\alpha_s}} = : 2 \xi\,.
		\end{equation}
		Let us denote 
		\begin{equation}\label{xi.def}
			u(t) := \Pi_{{\tilde\Lambda}} \left(\gamma(t) - p_1\right) \quad \forall t\in [0, t_\star]\,.
		\end{equation}
		{{If} $s=n$, then $u(t) \equiv \gamma(t) - p_1$ and \eqref{t.star} immediately {yields} $\|u(t_\star)\|= 2 \xi \geq \xi$. We now analyze the case $s< n$.} Since {for all $t$ one has} $\gamma(t ) \in F^{(s)}_{\Lambda} \subset (p + {\tilde\Lambda}_\reals )_{d_s}$ and $p_1 \in (p + {\tilde\Lambda}_\reals )_{d_s},$ {we can write}
		\begin{equation}\label{sbordamento}
			\begin{aligned}
				\Vert \Pi_{{\tilde\Lambda}}^\bot\left(\gamma(t) - p_1 \right)\Vert &\leq \Vert \Pi_{{\tilde\Lambda}}^\bot\left(\gamma(t) - p \right)\Vert + \Vert \Pi_{{\tilde\Lambda}}^\bot\left(p - p_1 \right)\Vert\\
				&= \operatorname{dist}\left(\gamma(t) - p, {\tilde\Lambda}_\reals \right) +\operatorname{dist}\left(p_1 - p, {\tilde\Lambda}_\reals\right)\\
				&= \operatorname{dist}\left(\gamma(t), p + {\tilde\Lambda}_\reals \right) + \operatorname{dist}\left(p_1, p + {\tilde\Lambda}_\reals \right) < 2 d_s\,,
			\end{aligned}
		\end{equation}
		{so that}
		\begin{equation}\label{magic.features}
			\begin{aligned}
				\Vert u(t_\star) \Vert &= \Vert \Pi_{{\tilde\Lambda}} (\gamma(t_\star) - p_1 ) \Vert\\
				&\geq \Vert \gamma(t_\star) - p_1 \Vert - \Vert \Pi_{{\tilde\Lambda} }^\bot (\gamma(t_\star) - p_1 ) \Vert \geq 2 \xi - {2}d_s \geq \xi\,,
			\end{aligned}    
		\end{equation}
		{where, {in the last inequality}, we have used \eqref{small.diam.1}.} {Thus, for all $s=1,\dots, n$, we have $\|u(t_\star)\| \geq \xi$.}
		Let $t_{**}$ be the smallest value in $[0, t_\star]$ such that $\|u(t_{**})\| = \xi.$\\
		{Then,} by the very definition of $\xi$ and $t_{**}$ and by continuity of the curve $u$, one has that
		\begin{equation}\label{intervals}
			\{ \| u(t)\|\ |\ t \in [0, t_{**}]\} = [0, \xi]\,.
		\end{equation}
		Therefore, since $h_0$ is P-steep, one has
                \begin{equation}\label{go.far}
			\begin{aligned}
				\max_{t \in [0, t_{**}]} \| \Pi_{{\tilde\Lambda}}\left(\omega(p_1 + u(t)) - \omega(p_1) \right)\| &\stackrel{\eqref{intervals}}{=} \max_{\substack{ \eta \in [0, \xi]\,, \\ \|u(t)\| = \eta}} \| \Pi_{{\tilde\Lambda}}\left(\omega(p_1 + u(t)) - \omega(p_1) \right)\|\\
				& \geq \max_{\eta \in [0, \xi]} \min_{\substack{u \in {\tilde\Lambda}_\reals\,,\\ \|u\| = \eta}}  \| \Pi_{{\tilde\Lambda}}\left(\omega(p_1 + u(t)) - \omega(p_1) \right)\|\\
				& \geq D_s \xi^{\alpha_s}\,.
			\end{aligned}
		\end{equation}
		Let then $\underline{t}$ be the point realizing the maximum in \eqref{go.far}; on the one hand, by \eqref{go.far}, one has
		$$
		\| \Pi_{{\tilde\Lambda}}\left(\omega(p_1 + u(\underline{t})) - \omega(p_1) \right)\|\geq D_s \xi^{\alpha_s}\,.
		$$
		But on the other hand, {by} using the fact that
		$$
		\| \gamma(\underline{t}) - (p_1 + u(\underline{t}))\| = \| \gamma(\underline{t}) - p_1 -\Pi_{\tilde\Lambda}( \gamma(\underline{t}) - p_1)\| = \|\Pi_{\tilde\Lambda}^\bot \left( \gamma(\underline{t}) - p_1 \right) \| \stackrel{\eqref{sbordamento}}{\leq} 2 d_s\,,
		$$
		and {by} recalling that $\gamma(\underline{t}) \in F^{(s)}_{\Lambda}(p)\subseteq Z^{(s)}_{\Lambda}\,,$ one also has
		$$
		\begin{aligned}
			\| \Pi_{\tilde\Lambda}\big(\omega(p_1 + u(\underline{t})) &- \omega(p_1)\big) \|  \leq	\| \Pi_{\tilde\Lambda}\left(\omega(\gamma(\underline{t})) - \omega(p_1)\right) \| + \|\omega(\gamma(\underline{t})) - \omega(p_1 + u(\underline{t})) \|\\ 
			& \leq 	\| \Pi_{\tilde\Lambda}\left(\omega(\gamma(\underline{t})) - \overline{\omega}\right) \| + 	\| \Pi_{\tilde\Lambda}\left( \omega(p_1) - \overline{\omega}\right) \| + \cC_\omega \|\gamma(\underline{t}) - (p_1 + u(\underline{t}))\|\\
			&\leq {2}
			\beta^{(s)}_\Lambda   + 2 \cC_\omega d_s \leq  3 \beta^{(s)}_\Lambda \,, 
		\end{aligned}
		$$
		{        where we used \eqref{small.diam}.}
		Recalling the definition of $\xi$ as in \eqref{t.star} one then gets {the contradiction}
		$$
		4 \beta^{(s)}_\Lambda 
		= D_s \xi^{\alpha_s} \leq \|\Pi_{\tilde\Lambda} \left(\omega(p_1 + u(\underline{t})) - \omega(p_1) \right) \| \leq  3 \beta^{(s)}_\Lambda \,.
		$$
		
	\end{proof}
	
	\begin{lemma}[{Small divisors estimate on $B^{(0)}$}]
		\label{small.div.0}
		Assume that the parameter $d_0$ defined at the
                beginning of Section \ref{rzone} fulfills
		\begin{equation}
			\label{imp.4}
			d_0\leq \frac{C_1}{2K^{A_1+1} \cC_\omega}\ ,
		\end{equation}
		{and take $p_0\in B^{(0)}$. Then, for any $k \in \mathbb{Z}^n \setminus
		\{0\}$ and $l \in \mathbb{Z}^m$} such that $\|l\| + \|k\|
		\leq K$, and for any $p \in F^{(0)}(p_0)$ (cf. \eqref{F.0}), one has
		{
			\begin{equation}\label{tutto.tranne.M.0}
				|\omega(p) \cdot k + \nu \cdot l | \geq \frac{C_1}{2K^{A_1}}\,.
			\end{equation}
		}
	\end{lemma}
	{
		\begin{proof} 
			Let $\Lambda$ be the maximal modulus  {of dimension $1$} containing $(k,l)$ and let
			$\overline\omega_\Lambda$ be its center.  {Note that, since $k \neq 0$ and $\|k\| + \|l\| \leq K$, $\Lambda$ is admissible.}
			By the very definition of $B^{(0)}$, for any $p_0 \in B^{(0)}$ one has that
			$$
			|\omega(p_0) \cdot k + \nu \cdot l| = |(\omega(p) - \ombar_\Lambda)\cdot k| = \|k\|\|\Pi_{\tilde\Lambda}(\omega(p_0) - \ombar_\Lambda)\| \geq \frac{C_1\|k\|}{K^{A_1}|\tilde\Lambda|} \geq \frac{C_1}{K^{A_1}}\,,
			$$
			where in the last passage we have used that $\|k\| \geq |\tilde\Lambda|$. Therefore for any $p \in F^{(0)}(p_0)$ one has
			\begin{align*}
				|\omega(p) \cdot k + \nu
				\cdot l |\geq |\omega(p_0) \cdot k + \nu \cdot l |-\left\|\omega(p)-\omega(p_0)\right\| \|k\| \geq
				\frac{C_1}{K^{A_1}}- \cC_\omega d_0 \geq \frac{C_1}{2 K^{A_1}}\,,
			\end{align*}
			which proves \eqref{tutto.tranne.M.0}.
	\end{proof}}
        {In the proof of the next lemma we will make use of the following {observation}.}

        {
        \begin{remark}
          \label{altra.base}
Let $\mathring{\Lambda}$ be
any $K$-submodule contained in $\Lambda$ with
same rank $s$ of $\Lambda$. We consider also the case in which $\mathring{\Lambda}$ 
is non maximal.
Then $\mathring{\Lambda} = \operatorname{span}_\Z\{
\mathring{v}_1, \dots, \mathring{v}_s\}$ with $\mathring{v}_j =
(\mathring{k}_j, \mathring{l}_j)$ and { $\|\mathring{k}_j\|+\|\mathring{l}_j\|\leq
K$.} 
Defining $\mathring{\cK} := [\mathring{k}_1 \dots \mathring{k}_s]$
the matrix with the components of the vectors $\mathring{k}_j$ as
columns, we have that $\mathring{\cK} = \cK A$ where $A$ is an
$s\times s$ matrix with integer coefficients. Then
	\begin{equation}
		\label{volume.2}
		\left|\tilde\Lambda\right| \leq \sqrt{\det(\mathring{\cK}^t \mathring{\cK})} \leq K^s\,,
	\end{equation}
	where the last passage follows from Hadamard's inequality.
        \end{remark}
}	
	
	\begin{lemma}[Small divisors estimate,  {$k \notin \tilde \Lambda_\R$}]\label{cor:small.divisors}
		Let $s = \{1, \dots, n-1\}$, let $\Lambda\in\cM^{(s)}_K$ and $p_0 \in B^{(s)}_{\Lambda}$. Define
		\begin{equation}
			\label{gamma.m}
			\gamma_{\Lambda}^{(s)} := \frac{C_{s+1}}{3 K^{A_{s+1}}|\tilde\Lambda|}
		\end{equation}
		and assume
		\begin{align}
			\label{le.puntine}
			\left(\frac{4\beta_{\Lambda}^{(s)}}{D_s}
			\right)^{\frac{1}{\alpha_s}} \leq \frac{\gamma_{\Lambda}^{(s)}}{4
				K \cC_\omega}
			\\
			\label{imp.2}
			A_{s+1} \leq A_s - 1\,,\quad
			{C_{s+1} \geq 6 C_s}\, .
		\end{align}
 For any
		$(k,l)\in\Z^n\times\Z^m$ s.t. $\|l\| + \|k\| \leq K$  {and $k \notin \tilde \Lambda_\R$}, and for any $p \in F^{(s)}_{\Lambda}(p_0)$, one has
		\begin{equation}\label{tutto.tranne.M}
			|\omega(p) \cdot k + \nu \cdot l | \geq \frac{\gamma^{(s)}_\Lambda}{2}\,.
		\end{equation}
	\end{lemma}
	\begin{remark}
		\label{cond.as}
		Equation \eqref{le.puntine} is equivalent to
		\begin{equation}
			\label{cond.as.2}
			\frac{C_{s+1}{D_s^{1/\alpha_s}}}{12 \cC_\omega(4C_s)^{1/\alpha_s}}\geq
			\frac{K^{A_{s+1}}K}{K^{A_s/\alpha_s}}\left|\tilde\Lambda\right|^{1-\frac{1}{\alpha_s}}\ ,
		\end{equation}
		which {is} a condition on the constants $C_s$ {if} the
		r.h.s. is bounded with respect to $K$ and $\Lambda$. Now, since $|\tilde\Lambda|\leq
		K^{s}$, the r.h.s. is
		bounded if
		\begin{equation}
			\label{cond.as.1}
			A_{s+1}+1-\frac{A_s}{\alpha_s}+s\left(1-\frac{1}{\alpha_s}\right)\leq
			0\ \qquad \iff\qquad A_{s+1}+(s+1)\leq \frac{A_s+s}{\alpha_s}\,.
		\end{equation}
		{We are going to impose {this} condition on the parameters $A_s$.}
		{Notice that \eqref{cond.as.1} implies the first inequality in \eqref{imp.2}.}
	\end{remark}

	\noindent
	{\it Proof of Lemma \ref{cor:small.divisors}.} The idea
	is to show that if \eqref{tutto.tranne.M} is violated, then $p$ is
	$s+1$ resonant, which contradicts the fact that it is in a block.
	
	We are going to proceed in three steps.
	\begin{itemize}
		\item {\sc Step 1:} First we identify a suitable lattice
		$\Lambda^+\in\cM^{(s+1)}_K$. 
		\item {\sc Step 2:}  Then we exploit the fact that, being $p_0 \in B^{(s)}_{\Lambda}$, in particular $p_0 \notin Z^{(s+1)}_{\Lambda^+}$, to  prove that
		$$
		|\omega(p_0)\cdot k + \nu \cdot l| \geq \gamma_M^{(s)}\,.
		$$
		\item {\sc Step 3:} Finally, we use Lemma
		\ref{lemma:diameters} to ensure that the lower bound for
		$|\omega(p_0)\cdot k + \nu \cdot l|$ obtained in the
		previous step implies the lower bound \eqref{tutto.tranne.M}
		for $|\omega(p) + \nu \cdot l|$.
	\end{itemize}
	{\sc Step 1}  {
		Let
		$v_j=(k_j,l_j)$, $j=1,...,s$ be a  {basis (over the integers) of} $\Lambda$ with
		$\left\|k_j\right\|+\left\|l_j\right\|\leq K$. We define
		$\Lambda^+$ as the $K$-maximal submodule of $\Z^{n}\times \Z^m$ containing $\operatorname{span}_\Z\left\{v_1,...,v_s,(k,l)\right\}$, and we {observe} that, since $k \notin \tilde{\Lambda}$, the modulus $\Lambda^+$ is admissible. Furthermore, {as} $\{v_1, \dots, v_s, (k,l)\}$ are linearly independent vectors in $\Lambda^+$,}
	by the very definitions of $\ombar_\Lambda$ and
	$\ombar_{\Lambda^+}$  {(see \eqref{defom.bar})}, one has
	$$
	\ombar_{\Lambda^+} \cdot k_j = - \nu \cdot
	l_j=\ombar_{\Lambda} \cdot k_j \,, \quad j = 1, \dots, s\,,
	$$
	so that
	\begin{equation}
		\label{pi.emme.zero}
		\Pi_{\tilde\Lambda} (\ombar_\Lambda -\ombar_{\Lambda^+}) = 0\,.
	\end{equation}
	{\sc Step 2}
	Let now $k_\parallel := \Pi_{\tilde\Lambda} k$ and $k_\bot := \Pi_{\tilde\Lambda}^{\bot} k$; one has
	\begin{align}\label{domande}
		|\omega(p_0) \cdot k + \nu \cdot l | & = |(\omega(p_0)
		- \ombar_{\Lambda^+}) \cdot k| \geq |(\omega(p_0) - \ombar_{\Lambda^+}) \cdot k_\bot| - |(\omega(p_0) - \ombar_{\Lambda^+}) \cdot k_\parallel|\,,
	\end{align}
	with
	\begin{equation}\label{concorsi}
		\begin{aligned}
			\vabs{(\omega(p_0) - \ombar_{\Lambda^+}) \cdot k_\parallel} &\leq \norm{\Pi_{\tilde\Lambda}(\omega(p) - \ombar_{\Lambda^+})} \norm{k_\parallel} = \norm{\Pi_{\tilde\Lambda}(\omega(p_0) - \ombar_{\Lambda})} \norm{k_\parallel}\\
			& \leq \frac{C_s}{K^{A_s} |\tilde\Lambda|} \norm{k_\parallel}\,,
		\end{aligned}
	\end{equation}
	and
	\begin{align*}
		\frac{|(\omega(p_0) - \ombar_{\Lambda^+}) \cdot
			k_\bot|}{\|k_\bot\|}&=\left\|(\Pi_{\tilde\Lambda^+}-
		\Pi_{\tilde\Lambda}) 
		(\omega(p_0)-\bar\omega_{\Lambda^+})\right\|  
		\\
		&\geq \| \Pi_{\tilde\Lambda^+} (\omega(p_0) - \ombar_{\Lambda^+})\| - \|\Pi_{\tilde\Lambda} (\omega(p_0) - \ombar_{\Lambda^+})\|\\
		&=\| \Pi_{\tilde\Lambda^+} (\omega(p_0) - \ombar_{\Lambda^+})\| - \|\Pi_{\tilde\Lambda} (\omega(p_0) - \ombar_{\Lambda})\|\\
		&\geq \frac{C_{s+1}}{K^{A_{s+1}}|\tilde\Lambda^+|} - \frac{C_s}{K^{A_s} |\tilde\Lambda|}\,,
	\end{align*}
	where in the last passage we have used the fact that $p_0 \in
	B^{(s)}_{\Lambda}$ and thus $p_0\not\in
	Z^{(s+1)}_{\Lambda^+}$.
        {Define now $\mathring{\cK}^+$ as the $s \times n$ matrix with columns $[k_1, \dots, k_s, k]$, where the vectors $k_1, \dots, k_s$ are defined as in the beginning of Step 1. By the maximality of $\Lambda^+$ and using \eqref{volume.2} with $\Lambda$ and $\mathring{\cK}$ {replaced, respectively,} by $\Lambda^+$ and $\mathring{\cK}^+$, one has
		\begin{equation}\label{volumi}
			|\tilde\Lambda^+| \leq \sqrt{\det ((\mathring{\cK}^+)^t \mathring{\cK}^+)}\leq |\tilde\Lambda|
			\|k_\bot\|\,.
		\end{equation} Therefore} we get
	\begin{equation}
		\label{ricorsi}
		\begin{aligned}
			|(\omega(p_0) - \ombar_{\Lambda^+}) \cdot
			k_\bot| \geq  \frac{C_{s+1}
				\|k_\bot\|}{K^{A_{s+1}}|\tilde\Lambda^+|} -
			\frac{C_s  {\|k_\bot\|}}{K^{A_s} |{\tilde \Lambda}|} \geq \frac{C_{s+1}}{K^{A_{s+1}}|\tilde\Lambda|} - \frac{C_s \|k_\bot\|}{K^{A_s} |\tilde\Lambda|}\,.
		\end{aligned}
	\end{equation}
{Then, by combining together formulas \eqref{domande}, \eqref{concorsi}, \eqref{ricorsi}, we obtain}
	\begin{align*}
		|(\omega(p_0) - \ombar_{\Lambda^+}) \cdot k| &\geq \frac{C_{s+1}}{K^{A_{s+1}}|\tilde\Lambda|} - \frac{C_s (\|k_\bot\| + \|k_\parallel\|)}{K^{A_s} |\tilde\Lambda|} \geq \frac{C_{s+1}}{K^{A_{s+1}}|\tilde\Lambda|} - \frac{2 C_s }{K^{A_s - 1} |\tilde\Lambda|}\,,
	\end{align*}
{so that, by taking hypothesis \eqref{imp.2} into account, we finally obtain}
	$$
	|(\omega(p_0) - \ombar_{\Lambda^+}) \cdot k| \geq \frac{C_{s+1}}{3 K^{A_{s+1}}|\tilde\Lambda|}\,.
	$$
	{\sc Step 3} Recall that, by Lemma \ref{lemma:diameters}, $\| p - p_0\| \leq 2 \left(\frac{4 \beta^{(s)}_\Lambda}{D_s}\right)^{\frac{1}{\alpha_s}}$; then  by Step 2 one has
	\begin{align*}
		|\omega(p) \cdot k + \nu \cdot l | & \geq |\omega (p_0) \cdot
		k + \nu \cdot l| - \|k\| \| p- p_0\|\cC_\omega
		\\
		&\geq \gamma_\Lambda^{(s)} - 2 K  \left(\frac{4 \beta^{(s)}_\Lambda}{D_s}\right)^{\frac{1}{\alpha_s}}\cC_\omega  \geq \frac{\gamma_\Lambda^{(s)}}{2}\,,
	\end{align*}
	where in the last passage we have used assumption \eqref{le.puntine}.
	\qed
	
	\begin{lemma}[Small divisors estimate,  {$k \in \tilde \Lambda_\R$}]\label{divisors.1}
		Let $s = \{1, \dots, n-1\}$, let $\Lambda\in\cM^{(s)}_K$. Define
		\begin{equation}
			\label{gamma.m.1}
			\tilde\gamma_{\Lambda}^{(s)} :=
			\frac{\Gamma}{|\tilde\Lambda|(s+1)^\tau K^{(s+1)\tau}}
		\end{equation}
		and assume
		\begin{align}
			\label{imp.2.1}
			A_{s} \geq (s+1)\tau+1\ ,\quad
			C_s\leq\frac{\Gamma}{2(s+1)^\tau}\ .
		\end{align}
For any
		$(k,l)\in\Z^n\times\Z^m\setminus\Lambda$ s.t. $\|l\| + \|k\| \leq K$  {such that $k \in \tilde \Lambda_\R \cap \Z^n$}, and for any $p \in Z^{(s)}_\Lambda$, one has
		\begin{equation}\label{tutto.tranne.M.1}
			|\omega(p) \cdot k + \nu \cdot l | \geq \frac{\tilde\gamma^{(s)}_\Lambda}{2}\,.
		\end{equation}
	\end{lemma}
	
	\proof
	We proceed in 2 steps: first we prove that {if $\omega_0$ belongs to the exact resonance $
	P_\Lambda$ of definition \ref{def:risonanza_esatta},} then
	\begin{equation}
		\label{omega0.1}
		\left|\omega_0\cdot k+\nu\cdot l\right|\geq
		\tilde\gamma^{(s)}_\Lambda\ ,
	\end{equation}
	then we prove that this implies an analogous estimate for $p\in
	Z^{(s)}_\Lambda$.

	 Let
	$v_j=(k_j,l_j)$, $j=1,...,s$ be a  {basis (over the integers) of} $\Lambda$ with
	{$\left\|k_j\right\|+\left\|l_j\right\|\leq K$.} We remark
	that there exist $c_1,...,c_s$ s.t.
	\begin{equation}
		\label{linear.eq}
		k=\sum_{j=1}^sc_jk_j\ ,
	\end{equation}
	so that one has
	\begin{equation}
		\label{lstar}
		(k,l)-\sum_{j=1}^sc_j(k_j,l_j)=(0,l^*)\ .
	\end{equation}
	We need to obtain some information of $l^*$. To this end we compute
	the numbers $c_j$. We first consider the matrix
	$\cK=[k_1...k_s]$ and remark that it has a minor of order $s$ with
	non-vanishing determinant. Without loss of generality we assume that it
	is $\check\cK:=[\check k_1...\check k_s] $ where $\check k_j$ is the
	truncation of $k_j$ to the first $s$ components. So, we {observe that
	solving \eqref{linear.eq}-\eqref{lstar}} is equivalent to solve
	\begin{equation}
		\label{linear.eq.1}
		\check k=\sum_{j=1}^sc_j\check k_j\ .
	\end{equation}
	By
	Cramer's theorem one has
	\begin{equation}
		\label{determinanti}
		c_1=\frac{\Delta[\check k\, \check k_2... \check
			k_s]}{\Delta[\check k_1... \check k_s] }\, ,..., \, c_s=\frac{\Delta[\check k_1... \check
			k_{s-1}\, \check k]}{\Delta[\check k_1... \check k_s] }
	\end{equation}
	where $\Delta:=\Delta[\check k_1... \check k_s]$ is the determinant of the
	matrix with column vectors $\check k_1... \check k_s $ and {an analogous notation holds} for
	the numerators in Eq. \eqref{determinanti}. By Hadamard's inequality,
	each one of the determinants has modulus bounded by
	$K^s$. Furthermore, {the Cauchy Binet formula ensures that} $|\Delta|\leq \left|\tilde\Lambda\right|$. We also know that
	$\Delta\in\Z$. From Eq. \eqref{determinanti}, one gets
	\begin{equation}
		\label{lstar.2}
		l^* =l-\sum_{j=1}^s\frac{\Delta_j}{\Delta}l_j\ ,
	\end{equation}
	where we denoted by $\Delta_j$ the $j-th$ determinant in the numerators
	of \eqref{determinanti}. {Now, take $\omega_0\in P_{\Lambda}$, so that $(\omega_0,\nu)\cdot (k_j,l_j)=0$ for all $j\in\{1,\dots,s\}$; then we have}
	\begin{equation}
		\label{omega.0}
		\left|\omega_0\cdot k+\nu\cdot
		l\right|=\left|(\omega_0,\nu)\cdot[(k,l)-\sum_{j=1}^sc_j
		(k_j.l_j)] \right|=|\nu\cdot l^*|\ , 
	\end{equation}
	but, {as $\nu$ is Diophantine, we can write}
	\begin{align}
		\label{diof.2}
		|\nu\cdot
		l^*|=\left|\nu\cdot\left(l-\sum_{j=1}^s\frac{\Delta_j}{\Delta}l_j
		\right) \right|
		=\frac{1}{\left|\Delta\right|}\left|\nu\cdot\left(l\Delta-\sum_{j=1}^s
		l_j\Delta_j\right) \right|\geq \frac{1}{\left|\Delta\right|}
		\frac{\Gamma}{[(s+1)K^{s+1}]^\tau}
		\\
		\nonumber
		\geq
		\frac{1}{\left|\tilde\Lambda\right|}
		\frac{\Gamma}{[(s+1)K^{s+1}]^\tau}
		\ ,
	\end{align}
	{where we have used the inequality $\left\|l\Delta-\sum_{j=1}^s
	l_j\Delta_j\right\|\leq (s+1)K^{s+1}$. This implies claim \eqref{omega0.1}}
	and concludes Step 1.
	
	Step 2. Let  $p\in Z^{(s)}_ \Lambda$ and let
	$\omega:=\omega(p)$, so that
	$$
	\left\|\Pi_{\tilde\Lambda}(\omega-\ombar_\Lambda)\right\|\leq
	\beta_{\Lambda}^{(s)}\ .
	$$
	Thus, if $\omega_0\in P_\Lambda$,  {recalling \eqref{def.om.1}} we have
	$$
	\left\|\Pi_{\tilde\Lambda}(\omega-\omega_0)\right\|\leq
	\left\|\Pi_{\tilde\Lambda}(\omega-\ombar)\right\|+ \left\|\Pi_{\tilde\Lambda}(\omega_0-\ombar)\right\|  { =\left\|\Pi_{\tilde\Lambda}(\omega-\ombar)\right\| } \leq
	\beta_{\Lambda}^{(s)}\ .
	$$
	It follows that
	$$
	\left|\omega\cdot k+\nu\cdot
	l\right|=\left|(\omega-\omega_0)\cdot k+ \omega_0\cdot k+\nu\cdot l
	\right| \geq \tilde\gamma^{(s)}_\Lambda-\beta^{(s)}_\Lambda K\geq \frac{\gamma^{(s)}_\Lambda}{2}\ ,
	$$
	provided $\tilde\gamma^{(s)}_\Lambda\geq \beta^{(s)}_\Lambda/2 K$,
	which is ensured by \eqref{imp.2.1}. This concludes Step 2.

	\begin{lemma}[{Non-overlapping of resonances for $\tilde \Lambda_\R = \tilde \Lambda'_\R$}]\label{lemma:nonOverlappingCenteredZones}
		Let $K\geq 1$, $s\in\{1,..,n\}$ and let $\Lambda  {\neq} \Lambda'$  {be two moduli in}
		$\MM_K^{(s)}$ such that $\tilde\Lambda_\R=\tilde\Lambda'_\R$.
		Assume that        
		\begin{equation}
			\label{cond.important}
			A_s\geq \tau(s+1)+1 \quad \text{and}\quad
			\frac{\Gamma}{4(s+1)^\tau C_s}\geq 1\ ,
		\end{equation}
		then one has 
		\begin{equation}\label{change.center}
			\dist{Z_{\Lambda}^{(s)}}{Z_{\Lambda'}^{(s)}} >
			\frac{\Gamma}{2\cC_\omega (s+1)^\tau
				K^{(s+1)\tau+1}}\min\left\{\frac{1}{|\tilde\Lambda|}\, ,\,
			\frac{1}{|\tilde\Lambda'|} \right\}\ .
		\end{equation}
	\end{lemma}
	\begin{proof}
		Let $\ombar$ and $\ombar'$ be the centers of the resonant planes
		associated to $\Lambda$ and $\Lambda'$ respectively. Denote
		$$
		\delta:=\frac{\Gamma}{2(s+1)^\tau\cC_\omega 
			K^{(s+1)\tau+1}}\min\left\{\frac{1}{|\tilde\Lambda|}\, ;\,
		\frac{1}{|\tilde\Lambda'|} \right\}\ .
		$$
		{Assume, for contradiction,} that there exist two points $p \in
		Z^{(s)}_{\Lambda}$, $p' \in Z^{(s)}_{\Lambda'}$ such that
		$$
		\left\|p-p'\right\|\leq \delta\ .
		$$
		We separate two cases, namely {the one in which}
		\begin{equation}
			\label{caso1}
			\frac{|\tilde\Lambda|}{|\tilde\Lambda'|}\leq
			\frac{\Gamma}{4C_s(s+1)^\tau} K^{A_s-\tau(s+1)-1}\ ,
		\end{equation}
		and the one in which {the  inequality above} is violated.
		
		We start by assuming \eqref{caso1}. Let $(k',l')\in\Lambda'$ be such
		that $\left\| k'\right\|+\left\| l'\right\|\leq K$ and $(k',l')\not \in
		 {\Lambda}$.  {Since by assumption $k \in \Lambda'_\R = \Lambda_\R$, we are in position to apply Lemma \ref{divisors.1} with $(k, l)$ replaced by $(k', l')$.} We {obtain}
		\begin{align*}
			\left\|\Pi_{\tilde\Lambda'}(\omega(p)-\ombar_{\Lambda'})\right\|\geq
			\frac{\left|(\omega(p)-\ombar_{\Lambda'})\cdot k'\right|}{\left\| k'\right\|}=
			\frac{|\omega(p)\cdot k'+\nu\cdot l'|}{\left\| k'\right\|}\geq
			\frac{\Gamma}{2|\tilde\Lambda|(s+1)^\tau K^{\tau(s+1)+1}}\ , 
		\end{align*}
		{so that we can also write}
		\begin{align*}
			\left\|\Pi_{\tilde\Lambda'}(\omega(p')-\ombar_{\Lambda'})\right\|=
			\left\|\Pi_{\tilde\Lambda'}(\omega(p')-\omega(p)+\omega(p)-\ombar_{\Lambda'})\right\|
			\\
			\geq
			\frac{\Gamma}{2|\tilde\Lambda|(s+1)^\tau
				K^{\tau(s+1)+1}}- \cC_\omega \left\|p-p'\right\|
			\geq \frac{\Gamma}{4|\tilde\Lambda|(s+1)^\tau
				K^{\tau(s+1)+1}}\ , 
		\end{align*}
		by the contradictory assumption. But, from \eqref{caso1}, we have that
		the r.h.s. is larger than $\beta^{(s)}_{\Lambda'}$, against the
		assumption that $p'\in Z^{(s)}_{\Lambda '}$. This concludes the
		argument in Case 1.
		
		For Case 2 we proceed exactly in the same way, but reversing the role
		of $\Lambda$ and $\Lambda'$. In this case we get a contradiction
		provided {that}
		\begin{equation}
			\label{caso2}
			\frac{|\tilde\Lambda'|}{|\tilde\Lambda|}\leq
			\frac{\Gamma}{4C_s(s+1)^\tau} K^{A_s-\tau(s+1)-1}\ .
		\end{equation}
		So the proof is completed if we {demonstrate} that \eqref{caso1} and
		\eqref{caso2} cover all the possible cases of ratios of the volume of
		the two modules. To prove this denote 
		$$
		x:=\frac{|\tilde\Lambda'|}{|\tilde\Lambda|}\ ,\quad
		y:=\frac{\Gamma}{4C_s(s+1)^\tau} K^{A_s-\tau(s+1)-1} \ ,
		$$
		then \eqref{caso1} and \eqref{caso2} can be rewritten respectively as
		$x\leq y$ and $x\geq 1/y$. Now any $x\in\R^+$ fulfills at least one of
		these equations if $y\geq 1$. But this is implied by condition \eqref{cond.important}.  
	\end{proof}
	
	\begin{lemma}[Non overlapping of resonances  {for $\tilde \Lambda_\R \neq \tilde \Lambda'_\R$}]
		\label{lemma:nonOverlappingResonances}
		Let $K\geq 1$ and $s\in\{1,..,n-1\}$ .
		Let $\Lambda,\Lambda' \in \MM_K^{(s)}$ be such that
		$\tilde \Lambda_{\R}\neq \tilde \Lambda'_{\R}$.
		{Assume that condition \eqref{imp.2} holds.}
	Then for any   $p \in B^{(s)}_{\Lambda}$ one has that
	\[
	\dist{\overline{F_{\Lambda}^{(s)}(p)}}{Z_{\Lambda'}^{(s)}} > 0\,.
	\]
\end{lemma}

\begin{proof}
	We divide our analysis in two cases:
	\begin{itemize}
		\item{\sc Case 1:} $\Lambda$ and $\Lambda'$ are such that
		\begin{equation}\label{m'.larger.than.m}
			\frac{|\tilde\Lambda|}{|\tilde\Lambda'|} < \frac{C_{s+1} K^{A_s - 1 - A_{s+1}}}{6 C_s}\,; 
		\end{equation}
		\item {\sc Case 2:} $\Lambda$ and $\Lambda'$ are such that
		\begin{equation}\label{case.2}
			\frac{|\tilde\Lambda|}{|\tilde\Lambda'|} \geq \frac{C_{s+1} K^{A_s - 1 - A_{s+1}}}{6 C_s}\,.
		\end{equation}
	\end{itemize}
	{\sc Case 1:} By assumption there exists $(k',l')\in\Lambda'$
	with {$\left\|k'\right\|+\left\|l'\right\|\leq K$}
	s.t. $k'\not\in \Lambda_\R$. 
	Then for any $\wt{p} \in
	F^{(s)}_{M, \ombar}(p)$, by Lemma \ref{cor:small.divisors}
	one has
	\begin{align*}
		\| \Pi_{\tilde\Lambda'} (\omega(\wt{p}) - \ombar_{\Lambda'})\|
		&\geq \frac{|(\omega(\wt{p})- \ombar_{\Lambda'}) \cdot
			k'|}{\|k'\|}=\frac{\left|\omega(\tilde p)\cdot k'+\nu\cdot
			l'\right|}{\|k'\|} \geq \frac{\gamma^{(s)}_{\Lambda}}{2K}
		\\ & = \frac{C_{s+1}}{6 K^{A_{s+1}+1}|\tilde\Lambda|} =
		\frac{C_{s+1} |\tilde\Lambda'|}{6 C_s|\tilde\Lambda|} K^{A_s -
			A_{s+1} - 1} \beta^{(s)}_{\tilde\Lambda'}\,.
	\end{align*}
	Now, since we are in Case 1, namely \eqref{m'.larger.than.m} holds, we deduce
	$$
	\| \Pi_{\tilde\Lambda'} (\omega(\wt{p}) - \ombar_{\lambda'})\| \geq  \frac{C_{s+1} |\tilde\Lambda'|}{6 C_s|\tilde\Lambda|}  K^{A_s - A_{s+1} - 1} \beta^{(s)}_{\Lambda'} > \beta^{(s)}_{\Lambda'}\,,
	$$
	namely $\wt{p} \notin Z^{(s)}_{\Lambda'}$.\\

	\noindent        {\sc Case 2:} By assumption 
	there exists $(k,l)\in\Lambda$
	with $\left\|k\right\|+\left\|l\right\|\leq K$
	s.t. $k\not\in \Lambda'_\R$.      Let
	$\left\{v'_j=(k'_j,l'_j))\right\}_{j=1}^s$ be a  {basis (over the integers) of } 
	$\Lambda'$ fulfilling
	$\left\|k'_j\right\|+\left\|l'_j\right\|\leq K$. Define
	$\Lambda^+$  {as the maximal $K$-modulus containing $\mathring{\Lambda}' := \operatorname{span}_\Z\left\{v'_1,...,v'_s,(k,l)\right\}$.} Consider
	the centers $\ombar_{\Lambda'}$ and
	$\ombar_{\Lambda^+}$; by their definition one has 
	\begin{equation}
		\label{pi.emmeprimo.zero}
		\Pi_{\tilde\Lambda'} (\ombar_{\Lambda'} -\ombar_{\Lambda^+}) = 0\,.
	\end{equation}
	We define
	\begin{equation}\label{def.nu}
		\tnu := \Pi_{\tilde\Lambda'}^\bot k\,.
	\end{equation}
	Note that,  {defining $\mathring{\cK}^+$ as the matrix with columns $[k'_1, \dots, k'_s, k]$, by \eqref{volume.2} with $\tilde \Lambda$ replaced by $\tilde \Lambda^+$ and $\mathring{\cK}$ replaced by $\mathring{\cK}^+$, one gets}
	\begin{equation}
		\label{ecco.perche}
		|\tilde\Lambda^+|  {\leq} \sqrt{\det \left( (\mathring{\cK}^+)^t \mathring{\cK}^+ \right)}\leq |\tilde\Lambda'| \|\tnu\|\,.
	\end{equation}
	{Furthermore $\tnu \in \tilde\Lambda^+$, since $\tilde\Lambda^+ \supset \tilde\Lambda'$,  $\tnu
		\in (\tilde\Lambda')^\bot$, and $\tnu \neq 0$, since $k \notin \tilde\Lambda'$ by construction. 
		Therefore, one has
		\begin{equation}\label{e.ovvio}
			\| \Pi_{\tilde\Lambda^+} v\|^2 =
			\|\Pi_{\tilde\Lambda'} v\|^2 + \frac{|v \cdot
				\tnu|^2}{\|\tnu\|^2}\ ,\quad \forall v \in
			\reals^n\,.
		\end{equation}
	}
	{We proceed by contradiction: assume {that there
			exists $\wt{p}\in F^{(s)}_{\Lambda}(p)$ such that}
		$\wt{p}\in Z^{(s)}_{\Lambda'}$; we then prove that $
		p\in Z^{(s+1)}_{\Lambda^+}$, which is absurd since $
		p\in B^{(s)}_{\Lambda}$, which does not intersect the zones of
		order $s+1$ by its very definition.  }
	
	{ First remark that, {since by definition $F^{(s)}_{\Lambda}(p)
			\subseteq Z^{(s)}_{\Lambda}$, and by construction we are taking
			$\wt{p} \in \overline{F^{(s)}_{\Lambda}(p)}$, we also have
			$\wt{p} \in \overline{Z^{(s)}_{\Lambda}}$}. Then under our
		contradictory assumption, using \eqref{pi.emmeprimo.zero} we have
		\begin{equation}\label{forse}
			\begin{aligned}
				\left|(\omega(\wt{p})-\bar\omega_{\Lambda^+})\cdot
				\tnu\right|&=\left|(\omega(\wt{p})-\bar\omega_{\Lambda^+})\cdot ( {k}-\Pi_{M'} {k})
				\right| \\
				&= \left|(\omega(\wt{p})-\bar\omega)\cdot k- \Pi_{\tilde\Lambda'}(\omega(\wt{p})-\bar\omega_{\Lambda'})\cdot k\right|  \\
				&\leq \beta^{(s)}_\Lambda
				K+\left\|\Pi_{\tilde\Lambda'}(\omega(\wt{p})-\bar\omega_{\Lambda'}
				)\right\|\left\| {k}\right\|\\
				& \leq(\beta^{(s)}_\Lambda
				+\beta^{(s)}_{\Lambda'})K\leq 2\beta^{(s)}_{\Lambda'}K\,,
			\end{aligned}
		\end{equation}
		where in the last passage we have used the fact that, since we are in Case 2,
		$$
		\beta^{(s)}_\Lambda \leq \frac{6 C_{s}}{C_{s+1}K^{A_s - 1- A_{s+1}}}
		\beta^{(s)}_{\Lambda'} \leq \beta^{(s)}_{\Lambda'}\,,
		$$
		and {we have also taken \eqref{imp.2} into account.}
		Therefore, {by} exploiting identities \eqref{e.ovvio} and \eqref{pi.emmeprimo.zero} and {by} the fact that $K/\left\|\tnu\right\|>1$, we obtain
		\begin{equation}	\label{overlappo}
			\begin{aligned}
				\left\|\Pi_{\tilde\Lambda^+}(\omega(\wt{p})-\bar\omega_{\Lambda^+})\right\|^2 & =
				\left\|\Pi_{\tilde\Lambda'}(\omega(\wt{p})-\bar\omega_{\Lambda^+})\right\|^2+\frac{\left|(\omega(\wt{p})-\bar\omega_{\Lambda^+})\cdot\tnu\right|^2}{\left\|\tnu\right\|^2}\\
				&
				\leq
				\left(\beta^{(s)}_{\Lambda'}\right)^2+\left(\frac{2K\beta^{(s)}_{\Lambda'}}{\left\|\tnu\right\|}\right)^2\\
				&\leq
				5\left(\frac{K}{\left\|\tnu\right\|}\beta^{(s)}_{\Lambda'}\right)^2
				=5 \left(\frac{C_s}{K^{A_s-1}|\tilde\Lambda'|\left\|\tnu\right\|}\right)^2\,.
			\end{aligned}
		\end{equation}
		We now recall that, by Lemma \ref{lemma:diameters} and assumption \eqref{le.puntine} one has
		$$
		\left\|\omega(\tilde p)-\omega(p)\right\|\leq
		\cC_\omega \|\wt{p} - p\| \leq \frac{\gamma^{(s)}_\Lambda}{2 K
		}\, ,
		$$
		 {with}
		\begin{equation*}
			\begin{aligned}
				\frac{\gamma^{(s)}_\Lambda}{2 K } \stackrel{\eqref{gamma.m}}{=}
				\frac{C_{s+1}}{6 C_s } K^{A_s - A_{s+1}-1} \beta^{(s)}_\Lambda
				\stackrel{\eqref{case.2}}{\leq} {\beta^{(s)}_{\Lambda'}}=\frac{C_s}{K^{A_s}\left|\tilde\Lambda'\right|}
				\\
				\leq \frac{C_s K}{ K^{A_s}\|\tnu\| |\tilde\Lambda'|}  {\leq}
				\frac{C_s}{K^{A_s-1}\left|\tilde\Lambda^+\right|}\,, 
			\end{aligned}
		\end{equation*}
		 {namely
			\begin{equation}\label{oppure.no}
				\left\|\omega(\tilde p)-\omega(p)\right\| \leq \frac{C_s}{K^{A_s-1}\left|\tilde\Lambda^+\right|}\,.
			\end{equation}
		}
		Then, combining  \eqref{overlappo} and \eqref{oppure.no} and using \eqref{ecco.perche}, we get
		\begin{equation}\label{giammai}
			\begin{aligned}
				\|\Pi_{\tilde\Lambda^+} (\omega(p) - \ombar_{\Lambda^+})\| &\leq \|\Pi_{\tilde\Lambda^+} (\omega(p) - \omega(\wt{p}))\| + \| \Pi_{\tilde\Lambda^+} (\omega(\wt{p}) - \ombar_{ {\Lambda^+}})\|\\
				&\leq \frac{C_s}{K^{A_s -1} {|\tilde \Lambda^+|}} + \frac{\sqrt{5} C_s}{K^{A_s - 1} |\tilde\Lambda'| \|\tnu\|} 
				\leq \frac{C_{s+1}}{K^{A_{s+1}}|\tilde\Lambda^+|}\,,
			\end{aligned}
		\end{equation}
		where in the last passage we have used again \eqref{imp.2}.
		Then Eq.  {\eqref{giammai}} shows that
		$p\in Z^{(s+1)}_{\Lambda^+}$, which contradicts the fact that $
		p$ is in the block $B^{(s)}_{\Lambda}$.}
\end{proof}

\subsection{Relationship among the constants.}
\label{le.costanti.s}

First we {choose} the constants $A_s$ in order to satisfy {the exact equality in condition
\eqref{cond.as.1}, that tunes the estimates on the small divisors}. By defining
\begin{equation}
	\label{Bs}
	B_s:=A_s+s\ ,\quad\forall s=1,...,n
\end{equation}
one gets the recursion $B_{s+1}=B_s/\alpha_s$ which leads to {the following relations:}
\begin{equation}
	\label{1.1}
	B_s:=\left[\prod_{i=1}^{s-1}\frac{1}{\alpha_i}\right]B_1\quad
	\Longrightarrow \quad
	A_s:=\left[\prod_{i=1}^{s-1}\frac{1}{\alpha_i}\right](A_1+1)-s\ ,\quad
	\forall s=2,...,n\ .
\end{equation}
{With this choice, the first condition in \eqref{imp.2}
is satisfied too.}

{Still with respect to the estimates on the small divisors,} in order to fulfill {hypothesis} \eqref{cond.as.2}, we take
\begin{equation}\label{quante.cose}
	C_{s+1} := 12 \cC_\omega \left(\frac{4
	C_s}{ {D_s}}\right)^{\frac{1}{\alpha_s}}\,, \quad
	\forall s=1,...,n-1\ ,
\end{equation}
which amounts to define {the quite cumbersome recursion}
\begin{equation}
	\label{Cs}
	C_s:=4^{\sum_{j=1}^{s-1}\prod_{i=j}^{s-1}\frac{1}{\alpha_i}}\left(12
	\cC_\omega \right)^{1+\sum_{j=2}^{s-1}\prod_{i=j}^{s-1}\frac{1}{\alpha_i}
	}  {\prod_{j = 1}^{s-1}\left( \frac{1}{D_j}\right)^{\prod_{i = j}^{s-1} \frac{1}{\alpha_i}} } C_1^{\prod_{i=1}^{s-1}\frac{1}{\alpha_i}}\ . 
\end{equation}
Note that, provided $C_1$ is small enough, \eqref{Cs} also guarantees {that}
\begin{equation}\label{Cs.small.1}
	C_s \leq \min\left\{\frac{1}{4}; \frac{\Gamma}{ {4}(s+1)^\tau}\right\} \quad
	\forall s=1,...,n\ ,\,
\end{equation}
and thus the validity of the second {condition \eqref{cond.important} on the non-overlapping of resonances, and of hypothesis
\eqref{imp.2.1} appearing in the small divisors estimates.} 
Then {due to \eqref{quante.cose},  {\eqref{Cs.small.1},} and to the fact that $\alpha_s \geq 1$, one easily checks that  the sequence $C_s$ is increasing in $s$, so that the second condition in \eqref{imp.2} is also automatically satisfied.}

We still have to fix $C_1$ and $A_1$. Consider the first {inequality in \eqref{imp.2.1}}.
{{As $A_s$ decreases with $s$} and the r.h.s. of \eqref{imp.2.1} {increases}, it is sufficient to choose
	\begin{equation}
		\label{Bd.choice}
		A_n := \tau(n+1) + 1\quad \Rightarrow \quad B_n= (\tau+1)(n+1)\,.
	\end{equation}
}
Therefore \eqref{cond.important} is satisfied for any $s$, and we can define
\begin{equation}
	\label{B1}
	B_1=\left[\prod_{i=1}^{n-1}\alpha_i\right]B_n\ ,
\end{equation}
which will determine the exponent $a$ controlling the stability times,
as well as the value of $A_1$.

In view of the application of the normal form lemma we also need to
compare the quantities $\tilde\gamma^{(s)}_\Lambda$ and $\gamma^{(s)}_\Lambda$  {defined in \eqref{gamma.m}, \eqref{gamma.m.1}}. We
remark that,  {using \eqref{imp.2.1} and \eqref{1.1},} one has
\begin{equation}
	\label{le.gamme}
	\tilde\gamma^{(s)}_\Lambda\geq \gamma^{(s)}_\Lambda\,,
\end{equation}
which will be useful in the next Section.

The radii $d_s$  {appearing in \eqref{def:salsicciotti}} will be chosen {suitably} in the next section.

\section{Analytic part}\label{analytic}

In the spirit of \cite{poe}, we start by giving the following definitions:
\begin{definition}\label{def:M_admissible}
	Given $\Lambda \subseteq \integers^{n}\times\Z^m$, $ {K'} \in \naturals$ and $\ta >0$, we say that a subset $D \subseteq \cU$ is \textbf{$(\ta,  {K'})$-nonresonant modulo $\Lambda$} if {for every} $p \in D$ one has
{	\begin{equation}
		|\omega(p)\cdot k + \nu \cdot l | \geq \ta \quad \forall (k, l) \in \integers^{n+m} \quad \text{s.t. } (k,l) \notin \Lambda\,, \quad \|k\|_1 + \|l\|_1 \leq  {K'}\,,
	\end{equation}
	where $\|\cdot\|_1$ indicates the standard $\ell^1$-norm.}
\end{definition}
\begin{definition} 	\label{definition:normalForm}
	A function $Z: \reals^n \times \torus^n \times \torus^m \to \reals$ is in \textbf{$\Lambda$-resonant normal form} if {its Fourier expansion reads}
	\[
	Z(p,q,\phi) = \sum_{(k,l) \in \Lambda} \hat{Z}_{kl}(p)
	e^{ik\cdot q+il\cdot \phi}\,.
	\]
\end{definition}

\begin{definition}
	\label{norme.analitiche}
	Let $D\subset\bC^n$, and, for $\sigma' > r>0$, {let $f$ be an analytic function
	on $D_{r}\times\torus^n_{2\sigma'}\times\torus^m_{2\sigma'}$, then we
	denote its Fourier norm by
	\begin{equation}
		\label{def.sigma}
		\left\|f\right\|_{r,\sigma'}:=\sup_{p\in
			D_r }\sum_{(k,l)\in \mathbb Z^{n+m}}\left|f(p)
		\right|e^{(\|k\|_1+\|l\|_1)\sigma'}\ .
	\end{equation}}
	If $r=\sigma'$ we will simply write $\left\|f
	\right\|_{\sigma'}:=\left\|f\right\|_{\sigma',\sigma'}$. 
\end{definition}

We recall the following lemma:
\begin{lemma}[Normal form lemma of \cite{poe}]\label{lem:lo.dice.poschel}
	Let $ {K'} \in \naturals,$ $\ta >0,\ r>0$,  {$\e'>0$} and $\Lambda \subseteq \integers^n\times\Z^m$. Suppose $D\subseteq \cU^{\mathbb{R}}_{\frac{ {\sigma'}}{2}}$ is $(\ta,  {K'})$-nonresonant modulo $\Lambda$, and that
	\begin{equation}\label{hyp:smallness}
		 {\e'} \le \frac{\ta r}{9 \cdot 2^7  {K'}}\,, \quad r \leq \min \left\lbrace \frac{8 \ta}{9  {K'} \cC_\omega}\,,  {\sigma'} \right\rbrace\,, \quad  {K'}  {\sigma'} \geq 6\,.
	\end{equation}
	Consider a Hamiltonian
	$$
	\widetilde{H}(P, Q) = \widetilde{h}_0(P) + V(P, Q)
	$$
	with $\widetilde{h}_0$ and $V$ analytic on $(D \times \reals^m)_{ {\sigma'}} \times \torus^{n+m}_{ {\sigma'}}$, and $\|V\|_{ {\sigma'}} \leq  {\e'}$.\\
	Then there exists a real analytic, symplectic {change of coordinates
	$$
	\Psi: (D \times \reals^m)_{\frac{r}{2}} \times \torus_{\frac{ {\sigma'}}{6}}^{n+m} \rightarrow (D \times \reals^m)_{r} \times \torus_{ {\sigma'}}^{n+m}
	$$
	satisfying the following properties: 
	\begin{itemize}
		\item[(i)]
		the composition reads
		\begin{equation}
			\widetilde{H} \circ \Psi = \widetilde{h}_0 + Z + V_*\ ,
		\end{equation} and $\widetilde{h}_0 + Z$ is in $\Lambda$-resonant normal form;
		\item[(ii)] $V_*$ is an exponentially small remainder (with respect to $K$), namely 
		$$
		\|V_*\|_{\frac{r}{2}, \frac{ {\sigma'}}{6}} \leq e^{- \frac{ {K' \sigma'}}{6}} {\e'}\ ;
		$$
		\item[(iii)] if $(P, Q):= \Psi(P', Q')$, then
		\begin{equation}
			\label{defo}
			\|P' - P\| \leq \frac{18  {K'\e'}}{\ta}\leq\frac{1}{2^6}r
		\end{equation}
		uniformly on $(D \times \reals^m)_{\frac r 2} \times \torus^{n+m}_{\frac{ {\sigma'}}{6}}$\,.
	\end{itemize}}
\end{lemma}

Since the Hamiltonian \eqref{hamiltoniana} is time dependent, in order to apply the Normal Form Lemma \ref{lem:lo.dice.poschel}, we reduce to the time independent case, namely we extend the phase space as usual, and we define
\begin{equation}\label{ham.no.time}
	\widetilde{H}(p, q, J, \phi) := h_0(p) + \nu \cdot J + \e V(p, q, \phi)\,, \quad J \in \reals^m\,,
\end{equation}
and we consider the associated Hamiltonian system with variables $(P, Q) := (p, J, q, \phi)$ and symplectic form $\Omega(P, Q) := d p \wedge d q + d J \wedge d \phi\,.$

We then can use Lemma \ref{lem:lo.dice.poschel} put the Hamiltonian
\eqref{ham.no.time} in normal form in each fast drift block.
First, we link $K$ with $\varepsilon$ as follows:
\begin{equation}
	\label{def:constantK}
	K:=\frac{1}{(G\varepsilon)^{\betas}}\ , 
\end{equation}
with $G\geq 1$ and $\betas$ two {parameters to be determined later in the paper}. Then we
apply Lemma \ref{lem:lo.dice.poschel} with different values of the parameters depending on
the dimension $s\leq n-1$ of the resonance modulus labeling the
block.  To this end we still have to choose the analyticity parameter $r=r_s$ and the diameter
$d_s$ in \eqref{def:salsicciotti}. We describe the choice we are going to do and then we give it.

 {In view of the fact that for any $\sigma > 0$ and $ u \in (D \times \mathbb{R}^m)_{\sigma} \times \mathbb{T}^{n + m}_\sigma$
\begin{equation}\label{norme.equiv}
	\|\hspace{-1pt}| u |\hspace{-1pt}\|_{\frac{\sigma}{2}} \leq \| u\|_{\frac{\sigma}{2}} \leq \coth^{n + m}\left(\frac{\sigma}{4}\right) \|\hspace{-1pt}| u |\hspace{-1pt}\|_{\sigma} < + \infty\,, \quad \|\hspace{-1pt}|u|\hspace{-1pt}\|_{\sigma} := \sup_{(D\times \mathbb{R}^m)_\sigma \times \mathbb{T}^{n+m}_\sigma} |u(\cdot)|\,,
\end{equation}
(see Appendix B of \cite{poe}), and recalling the equivalence of norms
$$
\|k\| + \|l\| \leq \|k\|_{\ell^1} + \|l\|_{\ell^1} \leq C_{n, m} (\|k\| + \|l\|)\,, \quad C_{n, m} = \max\{\sqrt{n}, \sqrt{m}\}\,
$$
that holds true $\forall (k, l) \in \Z^{n + m}$, we choose
\begin{equation}\label{parametri.nf}
	\sigma' = \frac{\sigma}{2}\,, \quad K' = C_{n,m} K\,, \quad \e' = \e \|V\|_{\frac{\sigma}{2}}\,.
	\end{equation}}
So, take $s\in\left\{0,...,n-1\right\}$, take $p_0\in
B^{(s)}_{\Lambda}$ for some $\Lambda\in\cM^{(s)}_K$, {and set} $D:=F^{(s)}_\Lambda(p_0)$.  We {can choose}
\begin{equation}
	\label{alphas}
	\ta_s := \frac{C_{s+1}}{{6}  {(C_{n, m}K)}^{A_{s+1} + s}}\ ,
\end{equation}
so that following {{Lemmas \ref{small.div.0},
	\ref{cor:small.divisors}, \ref{divisors.1} and formula \eqref{le.gamme}},}  the unperturbed Hamiltonian is $(\ta_s,  {K'})$-nonresonant modulo
$\Lambda$ in $D$ {, with $K' = C_{n,m} K$.} 

By the second of \eqref{hyp:smallness}, then the radius $r_s$ must
fulfill
\begin{equation}
	\label{rs.1}
	r_s\leq \frac{4C_{s+1}}{{27} \cC_\omega}\frac{1}{ {(C_{n,m}K)}^{B_{s+1}}}\ ,\quad
	\forall s=0,...,n-1\ .
\end{equation}
As we will see in a while the dependence of this parameter on $K$ will
directly affect the exponent  $\betas$, which controls the time
interval over which Nekhoroshev's theorem is valid. For this reason we
keep it as sharp as possible. So we take
\begin{equation}
	\label{rs.vera}
	r_s:=\frac{L_s}{K^{B_{s+1}}}\ ,\quad
	\forall s=0,...,n-1\ ,
\end{equation}
with a constant $L_s$ on which we impose a few conditions. Of course
the first condition comes from \eqref{rs.1}.  The other conditions
come from the use {that} we will make of the normal form Lemma. {Namely, we} will use
it to show that, for any solution with initial datum $p_0\in
B^{(s)}_\Lambda  {\cap \cU^\R}$, the actions do not leave the fast drift block
$F^{(s)}_\Lambda(p_0)$ in the directions orthogonal to
$\tilde\Lambda_\R$. The main condition we need for this is that the
deformation in the action variables, as estimated by Eq. \eqref{defo}
is smaller than $d_s/3$, so we choose ${\frac{1}{2}}
\frac{d_s}{3}=\frac{r_s}{2^6}$, namely
\begin{equation}
	\label{ds.1}
	d_s:=\frac{3}{2^5}r_s\equiv \frac{3}{2^5}\frac{L_s}{K^{B_{s+1}}}\quad
	\forall s=0,...,n-1\ .
\end{equation}
We also define
\begin{equation}
	\label{dd}
	d_n:=\frac{C_n}{2\cC_\omega K^{B_n}}\ ,
\end{equation}
{in order to satisfy condition \eqref{small.diam} for $s=n$.} We have now to impose {conditions} \eqref{small.diam} and \eqref{small.diam.1} {bounding the sizes of the fast drift blocks} {for all $s=1, \dots, n-1$,}
which leads to the final choice
\begin{equation}
	\label{fs}
	L_s:=\min\left\{
	\frac{4}{{27}}\frac{C_{s+1}}{\cC_\omega  {C_{n,m}^{B_{s+1}}}};\frac{{2^4}}{3}\frac{C_s}{ {\cC_\omega }};\frac{{2^4}}{3}\left(\frac{4C_s}{D_s}\right)^{1/\alpha_s}
	\right\}\ , \quad
	\forall s=0,...,n-1\ .
\end{equation}
So, {by}  {combining \eqref{alphas}, \eqref{rs.vera}}  {and \eqref{norme.equiv}} the smallness condition given by the  {first} {inequality} in
\eqref{hyp:smallness} becomes
\begin{equation}
	\label{smalln}
	\varepsilon\left\| V\right\|_{ {\frac{\sigma}{2}}}{\le}\frac{L_sC_{s+1}}{{27}\cdot 2^{8} {C_{n,m}^{2 B_{s+1}}}}\frac{1}{K^{2B_{s+1}}}\ ,\quad
	\forall s=0,...,n-1\ ,
\end{equation}
which, {due to} \eqref{def:constantK}, is fulfilled if $2B_{s+1}\betas\leq 1$ and
\begin{equation}
	\label{Gs}
	G\geq {\left(\frac{27 \cdot 2^{8} {C_{n,m}^{2 B_{s+1}}} \left\| V\right\|_{ {\frac{\sigma}{2}}} }{ {3} L_sC_{s+1}} \right)^{\frac{1}{2 \beta B_{s+1}}}} \quad
	\forall s=0,...,n-1\ .
\end{equation}

Thus we  get the following Proposition which is already formulated in
terms of the time dependent system:

\begin{proposition}\label{prop:normal.form}
	{Consider the Hamiltonian ${H}$} defined in
	\eqref{hamiltoniana}. Let $s=0,...,n-1$, $\Lambda\in\cM^{(s)}_K$, let 
	$p_0\in B^{(s)}_{\Lambda}$, {and define}
	$$
	D:=F^{(s)}_{\Lambda}(p_0)\ , \quad r_s:=\frac{L_s}{K^{B_{s+1}}}
	$$
	with $L_s$ given by \eqref{fs}, {and $K$ as in}
	\eqref{def:constantK}, with $G$ fulfilling \eqref{Gs}. 
	{Moreover, assume that}
	\begin{align}
		\label{cond.su.beta}
		&2B_{s+1}\betas \leq 1\,.
	\end{align}
	Then, {there exists $\varepsilon_*>0$ s.t., if
		$0<\varepsilon<\varepsilon_*$ then there exists}
	a real analytic, symplectic
	coordinates transformation $\Psi^{(s)}_{\Lambda}:D_{\frac{r_s}{2}}\times
	\torus^{n+m}_{\frac{\sigma}{ {4}}}\to D_{{r_s}}\times
	\torus^{n+m}_{ {\frac{\sigma}{2}} } $
	such that
	\begin{equation}
		\label{NF}
		{H} \circ \Psi^{(s)}_{\Lambda} = h_0  + Z + V_*
	\end{equation}
	has the following properties:
	\begin{itemize}
		\item[(i)] $Z$ is in $\Lambda$-resonant normal form
		\item[(ii)] $\|V_*\|_{\frac{r}{2}, \frac{\sigma}{ {12}}}
		\leq \left\| V\right\|_{ {\frac{\sigma}{2}}}  e^{-\frac{C_{*}}{\e^\betas}}\e$, with $ {C_{*} = \dfrac{C_{m,n}G^\betas\sigma}{12}}$
		\item[(iii)] If $(p, q,\phi):=
		\Psi^{(s)}_{\Lambda}(p', q',\phi')$, then 
		\begin{equation}
			\label{defo.s}
			\|p'-p\|\leq \frac{d_s}{6}
		\end{equation}
		{where $d_s$ is the quantity} defined in \eqref{ds.1}.
	\end{itemize}
\end{proposition}

{
  \begin{remark}
  \label{la_stella}
The threshold $\varepsilon_*$ is determined by the following conditions: 
\[
 {C_{n,m}}K\sigma\ge  {12},\qquad r_s\le  {\frac{\sigma}{2}},\qquad C d_s< {\frac{\sigma}{2}},\qquad \forall s=0,\ldots,n-1
\]
and
\[
\varepsilon\|V\|_{ {\frac{\sigma}{2}}}
{\le}
\frac{L_s C_{s+1}}{27\cdot 2^8   {C_{n,m}^{2 B_{s+1}}} K^{2B_{s+1}}},\quad \forall
s=0,\ldots,n-1\ ,
\]
with $K=(G\varepsilon)^{-\beta}$ and $G$ given by \eqref{Gs}.
\end{remark}}

\section{Dynamical argument}\label{dyn.arg}

Along this section we study the solution $p(t)$ of the Hamilton
equations of \eqref{hamiltoniana}, when the initial datum $p_0$ is in a 
resonant block $B^{(s)}_{\Lambda}$. {We will} use the notation $a\sleq b$ to mean ``there exists a constant $C$
independent of $\varepsilon$ and of all the {quantities that will
  eventually depend $\varepsilon$}
s.t. $a\leq Cb$'';
 if $a\sleq b$ and $b\sleq a$ we will write $a
\simeq b$. {Furthermore we will always assume that $\varepsilon$ is
	so small that one can apply Proposition \ref{prop:normal.form} in
	anyone of the blocks $F^{(s)}_\Lambda$.}

\begin{lemma}\label{Z0}
	{Consider a point} $p_0 \in Z^{(0)}\cap\cU$  {and let $C_*$ be the constant in Proposition \ref{prop:normal.form}}. Then, for ${\beta} \leq (2B_1)^{-1}$ and
	\begin{equation}
		\label{t0}
		|t|\leq \frac{d_0\sigma}{ {72}\varepsilon \left\|V\right\|_{ {\frac{\sigma}{2}}} }\exp\left(
		\frac{C_{*}}{\varepsilon^{\beta}}\right)\,,
	\end{equation}
	one has
	$$
	\|p(t)-p_0\| \leq \frac{d_0}{2}\sleq \varepsilon^{\beta B_1}\,.
	$$
\end{lemma}

\begin{proof}
	Set $D=F^{(0)}(p_0)$.
		Then, by Proposition \ref{prop:normal.form}, there exists a
	{symplectic change of coordinates} $\Psi_0 :
	D_{\frac
		{r_0} 2} \times \torus^{n+m}_{\frac{\sigma}{ {12}}} \rightarrow
	D_{r_0}
	\times \torus^{n+m}_{ {\frac{\sigma}{2}}}$ {that} puts the Hamiltonian
	\eqref{hamiltoniana} in the {resonant normal} form \eqref{NF}, with $Z$
	independent of $q$.
	 Define, as above, $(p, q, \phi) :=
	\Psi(p',  q', \phi')$. Then one has
	\begin{equation}
		\label{esp.0}
		\left|\frac{d}{dt} p'(t)\right| =  \left|\{V_*, p'\}\right|
		\leq\left\| V\right\|_{ {\frac{\sigma}{2}}} \frac{ {12}}{\sigma} \e e^{- {C_{*}} \e^{-\beta}}\,.
	\end{equation}
By the usual bootstrap argument it follows that $p'(t)\in D$.
	{Estimate \eqref{defo.s} ensures that}
	\begin{equation}
		\label{defo.0}
		\left\|p_0'-p_0\right\|\leq \frac{d_0}{6}\ ,
	\end{equation}
	{so that, provided $p'(t)\in D$, for all times satisfying the bound in \eqref{t0}, one has}
	\begin{equation}
		\label{d.0.final}
		\|p'(t) - p_0'\| \leq \frac{d_0}{6}\,.
	\end{equation}
	If this is true, then using again
	\eqref{defo.s} we get
	\begin{equation*}
		\|p(t) - p_0\| \leq \|p(t) - p'(t)\| + \|p'(t) - p'_0\| + \|p'_0 - p_0\|
		\leq 3\frac{d_0}{6}=\frac{d_0}{2}\ .
	\end{equation*}
\end{proof}

\begin{lemma}\label{lemma:quo.vadis}
	{Consider $s = 1, \dots, n-1$, $\Lambda\in\cM^{(s)}_K$ and $p_0 \in B^{(s)}_{\Lambda}\cap\cU$. Assume that
	$\betas\leq (2B_{s+1})^{-1}$, and denote} by $t_e$ the possibly infinite
	escape time of $p(t)$ from 
	$F^{(s)}_{\Lambda}(p_0)$. {Then, one has the following dichotomy}
	
{	\begin{enumerate}
		\item[(i)]either
		\begin{equation}
			\label{ts}
			|t_e| \geq \frac{d_s\sigma}{ {72}\varepsilon \left\|V\right\|_{ {\frac{\sigma}{2}}}}\exp\left(
			\frac{C_{*}}{\varepsilon^\betas}\right)\sgeq \varepsilon^{\beta B_{s+1}-1}\exp\left(
			\frac{C_{*}}{\varepsilon^\betas}\right)\,;
		\end{equation}
		\item[(ii)] or
		\begin{equation}\label{puffetta}
			p(t_e) \in B^{(s')}_{\Lambda'} \quad \text{ for some }
			\Lambda'\in \cM^{(s')}_K \text{ with } s' < s\,.
		\end{equation}
	\end{enumerate}}

\end{lemma}
\begin{proof}
	{We} suppose that {$|t_e|$} is smaller than the bound in \eqref{ts}; we shall prove that
	\eqref{puffetta} holds.  Recall the definition
	\eqref{def:salsicciotti} of $F^{(s)}_{\Lambda}(p_0)$, and
	{observe first that since the solution is a connected curve, the
	solution cannot escape the connected set  $D:=F^{(s)}_{\Lambda}(p_0)$ without touching its boundary}. Thus we
	have to analyze two possibilities, namely
	\begin{itemize}
		\item[1.] $p(t_e) \in \partial\{p_0 + \tilde\Lambda_\reals\}_{d_s}$
		\item[2.] $p(t_e) \in \{p_0 + \tilde\Lambda_\reals\}_{d_s}$, {and} $p(t_e) \in
		\partial Z^{(s)}_{\Lambda}$.
	\end{itemize}
	{We claim that Case 1 cannot happen. To see this,} we apply Proposition
	\ref{prop:normal.form}. To use the normal form, let $\lambda$
	be any vector in $\reals^n$ such that $\|\lambda\| = 1$ and
	$\lambda \bot \tilde\Lambda_\R$. Then, defining $f_\lambda(p) := (p-p_0) \cdot \lambda$,
	one has
	$$
	\text{dist}\left(p,\left\{p_0+\tilde\Lambda_{\R}\right\}\right)=\sup_{\lambda}\left|f_\lambda(p)\right|\ .
	$$
	Furthermore, {due to} Definition \ref{definition:normalForm} one has
	$$
	\left\{f_\lambda(p'),Z(p',q',\phi')\right\}= \sum_{(k,l)\in \Lambda} i(\lambda\cdot k)
	Z_{kl}(p')e^{ik\cdot q'+i\phi'\cdot l}\ ,
	$$  
	{as} $k$ varies in $\tilde\Lambda$ and
	$\lambda$ is orthogonal to $\tilde\Lambda$. Thus, proceeding as in the proof of Lemma \ref{Z0} we have
	$$
	\frac{d}{dt} f_\lambda(p') = \{ p' \cdot \lambda, h_0 + Z  + V_*\} = \{ p' \cdot \lambda, V_*\}\,,
	$$
	so that
	$$
	\left| \frac{d}{dt} f_\lambda(p')\right| \leq
	\frac{ {12}}{\sigma} e^{-C_{*} \e^{-\betas}}\left\| V\right\|_{ {\frac{\sigma}{2}}} \e\,,
	$$
	and by \eqref{defo.s}
	$$
	\begin{aligned}
		| f_\lambda(p(t))-f_\lambda(p_0) | &\leq |
		f_\lambda(p'(t))-f_\lambda(p(t)) |+|
		f_\lambda(p'(t))-f_\lambda(p'_0) |+|
		f_\lambda(p'_0)-f_\lambda(p_0) |
		\\
		& \leq \frac{d_s}{2}\ ,
	\end{aligned}
	$$ if $t$ violates \eqref{ts}. It follows that the distance { from $p(t)$ to $p_0+\tilde\Lambda_\R$ is strictly smaller than $d_s$, and this excludes Case
	1.}
	Therefore, Case 2 must hold. Then, by Lemma
	\ref{lemma:nonOverlappingCenteredZones} $p(t_e)$ cannot belong
	to $Z^{(s)}_{\Lambda'}$ for some $\Lambda'$ with $\tilde\Lambda'_{\R}=\tilde\Lambda_{\R}$, and by Lemma
	\ref{lemma:nonOverlappingResonances} $p(t_e)$ cannot belong to
	$Z^{(s)}_{\Lambda'}$ for any $\Lambda'\in\cM^{(s')}$ with $s'
	= s$. In particular,
	\begin{equation}\label{no.s}
		p(t_e) \notin \bigcup_{\Lambda' \atop \Lambda'\in\cM^{(s)}_K}
		Z^{(s)}_{\Lambda'}\,. 
	\end{equation}
	But then by \eqref{inclusive} it follows that
	$$
	p(t_e) \in B^{(s')}_{\Lambda'} \quad \text{ for some }
	\Lambda'\in\cM^{(s')}\ ,\quad s'  < s\,,
	$$
which {yields} the thesis.
\end{proof}
By a simpler, but similar argument which does not require the use of
the normal form lemma, one gets also the following lemma.

\begin{lemma}\label{s=d}
	Let $\Lambda\in\cM^{(n)}_K$ and $p_0 \in
	B^{(n)}_{\Lambda}\cap\cU$. Denote by $t_e$ the possibly infinite escape
	time of $p(t)$ from $F^{(n)}_{\Lambda}(p_0)$, then either
	$|t_e|=\infty$ or
	\begin{equation}\label{puffetta.1}
		p(t_e) \in B^{(s')}_{\Lambda'} \quad \text{ for some }
		\Lambda'\in\cM^{(s')}_K\  \text{ with } \  s' < n\,.
	\end{equation}
\end{lemma}

 {
\begin{proof}
 The proof follows along the same lines of Lemma \ref{lemma:quo.vadis}. Simply observe that, by its very definition, one has $F^{(n)}_{\Z^n}(p) = [Z^{(n)}_{\Z^n}]^p$. Therefore the only possible option if $t_e$ is finite is that $p(t_e) \in \partial Z^{(n)}_{\Z^n}$, namely only Case 2 in the proof of Lemma \ref{lemma:quo.vadis} can happen.
\end{proof}
}
We are now ready to control the dynamics over exponentially long times.

\begin{lemma}
	\label{daquiali}
	{Consider $s\in\left\{0,..., n\right\}$ and assume
	\begin{equation}
		\label{cond.final}
		\beta\leq \frac{1}{2B_1}\ ,
	\end{equation}
	and \eqref{Gs}} for all $s'\leq s$. Let $\Lambda\in\cM^{(s)}_K$ and $p_0\in B^{(s)}_{\Lambda}$, then one
	has
	\begin{equation}
		\label{stima}
		\|p(t) - p_0\| \leq C^+_s \e^{\frac{\betas A_s}{\alpha_s}}\ ,
		\quad \forall |t| \leq \frac{{d_{0}}\sigma}{ {72}\varepsilon \left\| V\right\|_{ {\frac{\sigma}{2}}}} \exp\left(\frac{C_*}{\varepsilon^\beta}\right)\,,
	\end{equation}
	with
	$$
	C^+_s := \left(2 \sum_{j=1}^{s} \left(\frac{4 C_j}{D_j} \right)^{\frac{1}{\alpha_j}} + \frac{C_1}{2 \cC_\omega }\right) G^{\frac{\betas {A_s}}{\alpha_s}}\,.
	$$
\end{lemma}
\begin{proof}
	{ Consider first the case $s\leq n-1$. We apply Lemma \ref{lemma:quo.vadis}, which requires
		$\beta<(2B_{s+1})^{-1}$.} Then
	{there are two possibilities}: 
	{
	\begin{enumerate}
		\item[(i)]either
		$p(t) \in F^{(s)}_{\Lambda}(p)$ for any $0<t \leq T_s$
		with  $$T_s:= \frac{d_s\sigma}{ {72}\varepsilon \left\| V\right\|_{ {\frac{\sigma}{2}}}} e^{C_{*} \e^{-\betas}}
		\ , $$
		so that, by Lemma \ref{lemma:diameters} one has
		\begin{equation}
			\label{faccio.s}
			\|p(t) - p_0\| \leq 2 \left(\frac{4}{D_s} \beta^{(s)}_M
			\right)^{\frac{1}{\alpha_s}} < C_s^+ \e^{\frac{\betas
					A_s}{\alpha_s}}\ , \quad \forall |t| \leq T_s\,,
		\end{equation}
		\item[(ii)] or, there exists a time $t_s$, which could be
		very short, s.t. $p(t_s)\in B^{(s')}_{\Lambda'}$ for
		some $\Lambda'\in \cM^{(s')}_K$
		{\bf with $\mathbf{s'<s}$}.
Then one can repeat the argument for the
		initial datum $p(t_s)$ with $T_s$ replaced by $T_{s'}$, and
		the r.h.s. of \eqref{faccio.s} with $s$ replaced by $s'$. Then
		one iterates. After at most $s$ steps one lands in $Z^{(0)}$
		where one can apply Lemma \ref{Z0} and show that once in
		$Z^{(0)}$ the solution does not
		move more than $d_0/2$. 
	\end{enumerate}
	}

	We remark that the iterative
	application of Lemma \ref{lemma:quo.vadis} and subsequently
	Lemma \ref{Z0} requires the condition
	$$
	\beta\leq \frac{1}{2B_{s'+1}}\quad \forall s'\leq s\ , 
	$$
	and this leads to \eqref{cond.final}
	
	The motion of $p(t)$ takes a time that we can bound from below
	by the shortest among the $T_{s'}$, $s'\leq s$, which is $T_0$.
	
	Concerning the distance traveled by the solution, it is
	bounded by the sum of the diameters of the blocks visited,
	which is bounded from above by
	\begin{equation}
		\label{somme.diam}
		\|p(t) - p_0\| \leq \sum_{j=0}^{s}\diam{F^{(s-j)}_{\Lambda_j}} \leq C^+_{s} \e^{\frac{\betas A_s}{\alpha_s}}\,.
	\end{equation}
	If the initial datum is in $B^{(n)}_{\Lambda}$ for some
	$\Lambda\in\cM^{(n)}_K$, then by Lemma \ref{s=d}
	one can repeat the argument, just adding the diameter of
	$F^{(s)}_{\Lambda}(p_0)$, {and obtains the same} thesis. 
\end{proof}

Taking $\beta=(2B_1)^{-1}$ we get the following corollary, {that allows to control the dynamics of orbits starting at any initial datum. }

\begin{corollary}
	\label{finito}
	Define
	\begin{equation}
		\label{tp1}
		\tp:=\prod_{i=1}^{n}\alpha_i\ ,\quad \tp_1:=\prod_{i=1}^{n-1}\alpha_i
	\end{equation}
	For any $p_0\in\cU$ one has
	\begin{equation}
		\label{finale.1}
		\left\| p(t)-p_0\right\|\leq C^+_n
		\varepsilon^{\frac{1}{2\tp}\frac{B_n-n}{B_n}} \ ,\quad
		|t|\leq \frac{{d_{0}}\sigma}{ {72}\varepsilon \left\| V\right\|_{ {\frac{\sigma}{2}}}}
		\exp\left(\frac{C_*}{\varepsilon^{\frac{1}{2\tp_1 B_n}}}\right)\ ,
	\end{equation}
{where we observe} that
	$$
	\frac{{d_{0}}\sigma}{ {72}\varepsilon \left\| V\right\|_{ {\frac{\sigma}{2}}}}\simeq
	{\varepsilon^{\beta B_1 - 1} \gtrsim \varepsilon^{-\frac 1 2}}\ .
	$$
\end{corollary}

\end{document}